\documentclass[11pt,a4paper,reqno,dvipsnames]{amsart}

\usepackage{amsfonts}
\usepackage{amsmath}
\usepackage{amssymb,stmaryrd}
\usepackage{amsthm}
\usepackage{amsxtra}
\usepackage{bbm}
\usepackage{braket}
\usepackage{bm}
\usepackage{comment}
\usepackage{extarrows}
\usepackage{graphicx}
\usepackage{latexsym}
\usepackage{mathrsfs}
\usepackage{mathtools}
\usepackage{yhmath}

\let\underbrace\LaTeXunderbrace

\usepackage{float}
\usepackage{booktabs}
\usepackage{multicol}
\usepackage{fullpage}
\usepackage{parskip}

\usepackage{xcolor}
\usepackage{verbatim}
\usepackage{listings}
\usepackage[normalem]{ulem}

\usepackage{tikz}
\usetikzlibrary{arrows}
\usetikzlibrary{decorations.pathreplacing}
\usetikzlibrary{graphs}
\usetikzlibrary{shapes.geometric}

\usepackage{pgfplots}
\pgfplotsset{compat=1.15}

\usepackage{natbib}
\setcitestyle{numbers,square,semicolon}

\usepackage[pdfborder={0 0 0}]{hyperref}
\hypersetup{colorlinks, citecolor=blue, linkcolor=blue, urlcolor=teal}

\usepackage{aliascnt}
\usepackage{cleveref}

\newif\ifdraft
\draftfalse 

\ifdraft
  \usepackage{refcheck}
  \makeatletter
  \newcommand{\refcheckize}[1]{%
    \expandafter\let\csname @@\string#1\endcsname#1%
    \expandafter\DeclareRobustCommand\csname relax\string#1\endcsname[1]{%
      \csname @@\string#1\endcsname{##1}\wrtusdrf{##1}}%
    \expandafter\let\expandafter#1\csname relax\string#1\endcsname
  }
  \makeatother
  \refcheckize{\cref}
  \refcheckize{\Cref}
\fi

\makeatletter
\newtheorem*{rep@theorem}{\rep@title}
\newcommand{\newreptheorem}[2]{%
  \newenvironment{rep#1}[1]{%
    \def\rep@title{#2 \ref{##1}}%
    \begin{rep@theorem}}%
  {\end{rep@theorem}}}
\makeatother

\newreptheorem{theorem}{Theorem}

\theoremstyle{plain}
\newtheorem{thm}{Theorem}[section]
\crefname{thm}{theorem}{theorems}
\Crefname{thm}{Theorem}{Theorems}

\newcommand{\newaliastheorem}[4]{%
  \newaliascnt{#1}{thm}%
  \newtheorem{#1}[#1]{#2}%
  \aliascntresetthe{#1}%
  \crefname{#1}{#3}{#4}%
  \Crefname{#1}{#2}{#4}%
}

\newaliastheorem{lem}{Lemma}{lemma}{lemmas}
\newaliastheorem{prop}{Proposition}{proposition}{propositions}
\newaliastheorem{cor}{Corollary}{corollary}{corollaries}
\newaliastheorem{clm}{Claim}{claim}{claims}
\newaliastheorem{assmp}{Assumption}{assumption}{assumptions}
\newaliastheorem{conj}{Conjecture}{conjecture}{conjectures}

\theoremstyle{definition}
\newaliastheorem{defn}{Definition}{definition}{definitions}
\newaliastheorem{ex}{Example}{example}{examples}
\newaliastheorem{notation}{Notation}{notation}{notations}
\newaliastheorem{remark}{Remark}{remark}{remarks}
\newaliastheorem{construction}{Construction}{construction}{constructions}
\newaliastheorem{terminology}{Terminology}{terminology}{terminologies}
\newaliastheorem{obs}{Observation}{observation}{observations}
\newaliastheorem{setup}{Setup}{setup}{setups}

\newtheorem{question}{Question}
\crefname{question}{question}{questions}
\Crefname{question}{Question}{Questions}

\crefname{problem}{problem}{problems}
\Crefname{problem}{Problem}{Problems}

\numberwithin{case}{thm}

\numberwithin{subcase}{case}

\usepackage[shortlabels]{enumitem}
\SetEnumerateShortLabel{a}{\textup{(\alph*)}}
\SetEnumerateShortLabel{A}{\textup{(\Alph*)}}
\SetEnumerateShortLabel{1}{\textup{(\arabic*)}}
\SetEnumerateShortLabel{i}{\textup{(\roman*)}}
\SetEnumerateShortLabel{I}{\textup{(\Roman*)}}

\DeclareMathOperator{\reg}{reg}
\DeclareMathOperator{\pdim}{pdim}
\DeclareMathOperator{\supp}{supp}

\DeclareMathOperator{\dist}{dist}
\DeclareMathOperator{\height}{height}
\DeclareMathOperator{\ord}{ord}

\begin{document}

\title{Realizable $(\reg,\deg h)$-Pairs for Cover Ideals via Independence Polynomials}

\author[Biermann]{Jennifer Biermann} 
\address[J.~Biermann]{Department of Mathematics and Computer Science, Hobart and William Smith Colleges, 300 Pulteney
St. Geneva, NY 14456, USA}
\email{\textcolor{blue}{\href{mailto:biermann@hws.edu}{biermann@hws.edu}}}

\author[Chau]{Trung Chau} 
\address[T.~Chau]{Chennai Mathematical Institute, Siruseri, Kelambakkam, Tamil Nadu 603103, India}
\email{\textcolor{blue}{\href{chauchitrung1996@gmail.com}{chauchitrung1996@gmail.com}}}

\author[Kara]{Selvi Kara} 
\address[S.~Kara]{Department of Mathematics, Bryn Mawr College, Bryn Mawr, PA 19010}
\email{\textcolor{blue}{\href{mailto:skara@brynmawr.edu}{skara@brynmawr.edu}}}

\author[O'Keefe]{Augustine O'Keefe} 
\address[A.~O'Keefe]{Department of Mathematics and Statistics, Connecticut College, 270 Mohegan Avenue Pkwy, New
London, CT 06320, USA}
\email{\textcolor{blue}{\href{mailto:aokeefe@conncoll.edu}{aokeefe@conncoll.edu}}}

\author[Skelton]{Joseph Skelton} 
\address[J.~Skelton]{Department of Mathematics, William \& Mary, Jones Hall, Williamsburg, VA 23185, USA}
\email{\textcolor{blue}{\href{mailto:jwskelton@wm.edu }{jwskelton@wm.edu }}}

\author[Sosa Castillo]{Gabriel Sosa Castillo} 
\address[G.~Sosa Castillo]{Department of Mathematics, Colgate University,
Hamilton, NY, 13346, USA}
\email{\textcolor{blue}{\href{mailto:gsosacastillo@colgate.edu}{gsosacastillo@colgate.edu}}}

\author[Vien]{Dalena Vien} 
\address[D.~Vien]{Department of Mathematics, Bryn Mawr College, Bryn Mawr, PA 19010}
\email{\textcolor{blue}{\href{mailto:dvien@brynmawr.edu}{dvien@brynmawr.edu}}}

\begin{abstract}
Let $G$ be a finite simple graph on $n$ vertices and set $R=\Bbbk[x_1,\dots,x_n]$, with edge ideal
$I(G)$ and cover ideal $J(G)$. We give an explicit description of the $h$-polynomial of $R/J(G)$, in a form that extends to the Alexander dual of any squarefree monomial ideal.  We then express
$\deg h_{R/I(G)}(t)$ and $\deg h_{R/J(G)}(t)$ in terms of the independence polynomial
$P_G(x)=\sum_{i\ge 0} g_i x^i$ via  an invariant  $M(G)$, the multiplicity of $x=-1$ as a root of $P_G(x)$. In particular, we prove
\[
\deg h_{R/I(G)}(t)=\alpha(G)-M(G)
\qquad\text{and}\qquad
\deg h_{R/J(G)}(t)=n-2-M(G),
\]
where $\alpha(G)$ is the independence number of $G$. As a corollary, $M(G)$ is the additive inverse of the $\mathfrak{a}$-invariants of $R/I(G)$ and $R/J(G)$. 
We develop recursions and closed formulas for $M(G)$ for broad graph families, and use them to analyze
which $(\reg,\deg h)$-pairs occur for cover ideals within chordal classes, including explicit constructions
realizing extremal behavior.
We conclude with a conjectural bound on $\bigl|\reg(R/J(G))-\deg h_{R/J(G)}(t)\bigr|$ for connected graphs.
\end{abstract}

\maketitle

\section{Introduction}

Let $G$ be a finite simple graph on $n$ vertices and set $R=\Bbbk[x_1,\dots,x_n]$.
Associated to $G$ are the edge ideal $I(G)$ and the cover ideal $J(G)$.
This paper studies the interaction between two invariants of $R/J(G)$ that have independently played
central roles in combinatorial commutative algebra: the Castelnuovo--Mumford regularity
$\reg(R/J(G))$ and the degree of the $h$-polynomial $\deg h_{R/J(G)}(t)$. 
Our emphasis throughout is on understanding $\deg h_{R/J(G)}(t)$ and how it constrains (and is constrained by) regularity.

Regularity measures the complexity of a minimal free resolution, while the $h$-polynomial controls the
Hilbert series and encodes subtle information about growth and asymptotics of the Hilbert function.
Despite their common origin in the graded structure of $R/J(G)$, these invariants are often studied
through different lenses, and a systematic understanding of their \emph{joint behavior} for fixed $n$
remains limited. For edge ideals, this joint behavior has been investigated in
\cite{biermann2024degree,hibi2022regularity,hibi2019regularity}, and more generally for monomial ideals
in \cite{hibi2018regularity}, for binomial edge ideals in \cite{hibi_binomial}, and for toric ideals of graphs in \cite{favacchio2020regularity}. In contrast, the corresponding picture for cover ideals appears to be
largely unexplored. Motivated by extensive computations, we ask:

\begin{quote}
For connected graphs on a fixed number of vertices, which pairs
$$\bigl(\reg(R/J(G)),\deg h_{R/J(G)}(t)\bigr)$$
can occur, and what structural features of $G$
force the pair to lie on (or away from) the diagonal and its neighboring lines?
\end{quote}

This question becomes more tractable once one brings in a 
purely combinatorial object: the independence polynomial
\[
P_G(x)=\sum_{i\ge 0} g_i x^i,
\]
where $g_i$ counts independent sets of size $i$ in $G$.

Independence polynomials have a long history in graph theory and statistical physics, and their roots
have been studied from many perspectives (\cite{bencs2018trees,brown2000roots, brown2004location, ChudnovskySeymour2007, CutlerKahl2016, peters2019conjecture}).  Here we focus on a single distinguished root, $x=-1$,
and introduce the invariant
\[
M(G)\ :=\ \text{the multiplicity of $x=-1$ as a root of $P_G(x)$}.
\]
To our knowledge, this invariant has not previously been studied as an invariant in its own right,
and it does not appear explicitly in the existing graph theory or commutative algebra literature.  
Our results suggest that $M(G)$ is a remarkably
effective ``bridge'' invariant: it packages a delicate cancellation phenomenon in $P_G(x)$ at $-1$
into an integer that directly controls degrees of $h$-polynomials for both $R/I(G)$ and $R/J(G)$.
In particular, $M(G)$ is sensitive to graph operations that preserve the combinatorics of independent
sets in a controlled way (such as whiskering), and it admits clean
recursions in families where regularity is already well understood (such as chordal graphs).
This makes $M(G)$ a natural new tool for translating between enumerative graph data and homological
invariants of edge and cover ideals.

Our starting point is an explicit description of the $h$-polynomial for cover ideals (\Cref{lem:HS1}), developed in
Section~\ref{sec2:cover_ideals}.  This result is an analogue of \cite[Theorem 3.2]{biermann2024degree}, which treats edge ideals. In addition, one can generalize both expressions for the edge and cover ideals to any squarefree monomial ideal and its Alexander dual. Lastly, in this section, we relate the degrees of the
$h$-polynomials for edge ideals and cover ideals since their $a$-invariants coincide (\Cref{thm:deg_edge_cover}).  

The key link to independence polynomials and the invariant $M(G)$ is developed in Section~\ref{sec:deg-h-roots-independence},
where we show that $\deg h$ is determined by  the
invariant $M(G)$ (\Cref{lem:deg-via-ord}), yielding formulas of the form
\[
\deg h_{R/I(G)}(t)=\alpha(G)-M(G) \qquad\text{and}\qquad \deg h_{R/J(G)}(t)=n-2-M(G).
\]
Thus, the problem of computing $\deg h$ is reduced to understanding $M(G)$.

The remainder of the paper develops methods to compute $M(G)$ and applies them to families that
populate (or obstruct) regions of the $(\reg,\deg h)$-plot for cover ideals. In Section~\ref{sec:rad2-trees} we obtain explicit
formulas for $\reg(R/J(G))$ and $\deg h_{R/J(G)}(t)$ for trees of radius at most two (\Cref{cor:radius2-realizable-pairs}).
To handle forests more generally, Section~\ref{sec:deg-h-recursion-forests} establishes a
 recursion for $M(G)$ (\Cref{thm:JK-recursion-M}), giving an effective reduction to induced subforests as an analogue of Jacques and Katzman's work on projective dimension from \cite{jacques2005betti}.
Section~\ref{sec:whiskering} studies whiskering operations and derives closed formulas for uniform
whiskering (\Cref{lem:q-whisker-indep-poly}), together with a vertex-whiskering formula expressed in terms of $G$ and $G\setminus v$ (\Cref{thm:one_whisker}).

We then turn to structured chordal classes with particularly rigid behavior.
In Section~\ref{sec:split-graphs} we focus on split graphs. Using a split decomposition
\(V(G)=C\sqcup I\), where \(C\) is a clique and \(I\) is an independent set, we obtain an explicit formula for \(P_G(x)\) and deduce closed formulas for \(\alpha(G)\) and \(M(G)\).
In addition, we show that for connected split
graphs the pair $\bigl(\reg(R/J(G)),\deg h_{R/J(G)}(t)\bigr)$ lies on the diagonal (\Cref{cor:split_cover}), with a precise
description of which diagonal values occur for fixed $n$ (\Cref{cor:split-realizable-pairs}).
Section~\ref{sec:reg=n-2} constructs chordal families satisfying $\reg(R/J(G))=n-2$ while
$\deg h_{R/J(G)}(t)$ varies, producing vertical strings in the $(\reg,\deg h)$-plot.
Section~\ref{sec:deg=reg-1} provides a further chordal construction on the line $\deg h=\reg-1$.
In contrast, Section~\ref{sec:block_graphs} proves a restriction for block graphs: connected block
graphs lie on or above the line $\deg h=\reg-1$, so this class cannot appear strictly below the diagonal (\Cref{thm:block-M-leq-i}).

Finally, in Section~\ref{sec:possible-values}, we compile computational data for connected graphs on $n$
vertices up to a fixed range, discuss the feasible region for $(\reg,\deg h)$-pairs, and
propose the bound
\[
\bigl|\reg(R/J(G))-\deg h_{R/J(G)}(t)\bigr|
\le \left\lceil\frac{n}{2}\right\rceil-2
\]
as \Cref{conj:bounds-reg-deg}.

\section{Preliminaries}\label{sec:prelim}

Throughout, let $G$ be a finite simple graph on vertex set $[n]=\{1,\dots,n\}$,
and we work in the standard graded polynomial ring $R=\Bbbk[x_1,\dots,x_n]$. 

\begin{defn}
The \emph{independence complex} of $G$, denoted $\Delta(G)$, is the simplicial complex on
$[n]$ whose faces are the independent sets of $G$.
\end{defn}

For a simplicial complex $\Delta$ on $[n]$, let $I_\Delta$ denote its Stanley--Reisner ideal.
The Alexander dual of $\Delta$ is
\[
\Delta^\vee:=\{ [n]\setminus F : F\notin \Delta\}.
\]
The Alexander dual of a squarefree monomial ideal $I$ is denoted $I^\vee$.

\begin{defn}
The \emph{edge ideal} of $G$ is
\[
I(G):=\bigl(x_ix_j:\{i,j\}\in E(G)\bigr)\subseteq R.
\]

The \emph{cover ideal} of $G$ is 
\[
J(G) := \Bigg(\prod_{i\in \mathcal{C}} x_i : \mathcal{C} \text{ a minimal vertex cover of } G\Bigg)\subseteq R.
\]
\end{defn}

\begin{remark}
The edge and cover ideals of $G$ are related through Alexander duality:
\[
J(G)=I(G)^\vee.
\]
Equivalently, $J(G)=I_{\Delta(G)^\vee}$ and
\[
J(G)=\bigcap_{\{i,j\}\in E(G)}(x_i,x_j).
\]
Note that, if $G$ has at least one edge, then $\height(J(G))=2$ and hence
\[
\dim(R/J(G))=n-2.
\]  
\end{remark}

\begin{defn}
 A \emph{graded minimal free resolution} of $R/J(G)$ is an exact sequence of the form
$$
0\longrightarrow \bigoplus_{j \in \mathbb{Z}} R(-j)^{\beta_{p,j}} \longrightarrow \bigoplus_{j \in \mathbb{Z}} R(-j)^{\beta_{p-1,j}} \longrightarrow \cdots \longrightarrow 
\bigoplus_{j \in \mathbb{Z}} R(-j)^{\beta_{0,j}} \longrightarrow R/J(G) \longrightarrow 0.
$$
 The exponents $\beta_{i,j} = \beta_{i,j}(R/J(G))$ are invariants of $R/J(G)$ called the \emph{graded Betti numbers} of $R/J(G)$.  The \emph{Castelnuovo-Mumford regularity}, denoted by $\reg (R/J(G))$, is  defined in terms of the non-zero graded Betti numbers of $R/J(G)$:
 \[
    \reg (R/J(G)) = \max \{ j-i : \beta_{i,j} (R/J(G)) \neq 0\}.
 \] 
\end{defn}

One can study the regularity of $R/J(G)$ through the projective dimension of $R/I(G)$ due to the following celebrated result of Terai from \cite{terai2007}. We use this result repeatedly throughout the paper.

\begin{remark}
    Regularity is dual to projective dimension.
    \[
    \reg (R/J(G))= \pdim I(G)=\pdim (R/I(G))-1.
    \]
\end{remark}

Projective dimension of edge ideals was studied in \cite{ProjDim} by Dao and Schweig. In this paper, we often focus on chordal graphs and the following result from \cite[Corollary 5.6]{ProjDim} is useful for our regularity computations.

\begin{thm}\label{thm:pdim_chordal}
Let $G$ be a chordal graph. Then
\[
\pdim(R/I(G))-1 = n - i(G)
\]
where $i(G)$ is the independent domination number of $G$.   
\end{thm}
 
The independent domination number $i(G)$ is the minimum size of an independent dominating set. It is also equal to the minimum cardinality among all maximal independent sets of $G$. 

\begin{defn}
 Let $A=\bigoplus_{d\ge 0}A_d$ be a standard graded $k$-algebra. Its Hilbert series is
\[
H_A(t):=\sum_{d\ge 0}HF_A(d)t^d,
\]
where  $HF_A(d)=\dim_k A_d $ is the Hilbert function of $A$ in degree $d$. It is well-known that
\[
H_A(t)=\frac{h_A(t)}{(1-t)^{\dim A}}
\]
for a unique polynomial $h_A(t)\in\mathbb{Z}[t]$, called the \emph{$h$-polynomial} of $A$.
The \emph{$\mathfrak{a}$-invariant} of $A$ is
\[
\mathfrak{a}(A):=\deg h_A(t)-\dim A.
\]   
\end{defn}

Next, we encode independent sets via a generating function.
Let $\alpha=\alpha(G)$ denote the independence number of $G$, and for each $i\ge 0$ let $g_i$
denote the number of independent sets of $G$ of size $i$ (so $g_0=1$, $g_1=n$, and $g_i=0$ for $i>\alpha$).

\begin{defn}\label{def:independence-poly}
The \emph{independence polynomial} of $G$ is
\[
P_G(x)=\sum_{i=0}^{\alpha} g_ix^i.
\]
\end{defn}

\begin{defn}\label{def:gG}
Following \cite{biermann2024degree}, define
\[
g(G):=\sum_{i=1}^{\alpha}(-1)^{i-1}g_i.
\]
\end{defn}

\begin{remark}\label{rem:eulerchar-indep}
Let $\Delta(G)$ be the independence complex of $G$. An independent set of size $i$
corresponds to a face of $\Delta(G)$ of dimension $i-1$, so $f_{i-1}=g_i$.
Hence the (unreduced) Euler characteristic of $\Delta(G)$ is
\[
\chi(\Delta(G))=\sum_{k\ge 0}(-1)^k f_k
=\sum_{i=1}^{\alpha}(-1)^{i-1}g_i.
\]
In particular, $g(G)=\chi(\Delta(G))$, and the reduced Euler characteristic is
$\widetilde\chi(\Delta(G))=g(G)-1$.
Moreover,
\[
P_G(-1)=\sum_{i=0}^{\alpha}g_i(-1)^i
=1-\sum_{i=1}^{\alpha}g_i(-1)^{i-1}
=1-g(G).
\]
Thus $g(G)=1 \iff \widetilde\chi(\Delta(G))=0 \iff P_G(-1)=0$.
\end{remark}

\begin{remark}\label{rem:h-from-independence}
Since $R/I(G)$ is the Stanley--Reisner ring of $\Delta(G)$ and $\dim\Delta(G)=\alpha-1$,
the relationship between the $f$- and $h$-vectors gives
\[
h_{R/I(G)}(t)=\sum_{i=0}^{\alpha} g_i t^i(1-t)^{\alpha-i}.
\]
Equivalently, in terms of the independence polynomial,
\begin{equation}\label{eq:h-from-independent-poly}
h_{R/I(G)}(t)=(1-t)^{\alpha}P_G\!\left(\frac{t}{1-t}\right).
\end{equation}
\end{remark}

The degree of the $h$-polynomial for edge ideals is studied in \cite{biermann2024degree}.
In particular, \cite{biermann2024degree} gives an explicit formula for $\deg h_{R/I(G)}(t)$
and characterizes when this degree attains its maximal possible value. More broadly,
the problem of determining when the degree of an $h$-polynomial of a squarefree monomial ideal is maximal
is well understood in the literature; see, for example, \cite{Hibi_book}. We recall the relevant results
from \cite{biermann2024degree} below.

\begin{thm}[{\cite[Theorem 3.2 and Corollary 3.4]{biermann2024degree}}]\label{thm:h_poly_edge}
Let $G$ be a finite simple graph. 
Then
\begin{equation}\label{eq:h_poly_edge}
h_{R/I(G)}(t)
=
\Bigl(1-\sum_{s=0}^{\alpha-1}D_s\Bigr)(1-t)^{\alpha}
+
\sum_{s=0}^{\alpha-1}D_s(1-t)^{\alpha-s-1}
\end{equation}
where
\[
D_s = \sum_{j=s+1}^{\alpha}(-1)^{j-1-s}\binom{j}{s+1}g_j.
\]
Moreover,
\begin{itemize}
    \item $\deg h_{R/I(G)}(t)=\alpha(G)$ if and only if $g(G)\neq 1$;
    \item if $g(G)=1$, then $\deg h_{R/I(G)}(t)=\alpha-d-1$, where
    \[
    d=\min\{s : D_s\neq 0\}.
    \]
\end{itemize}
\end{thm}

\begin{remark}\label{rem:gen_h_poly}
The proof of~\cite[Theorem 3.2]{biermann2024degree} extends verbatim to an arbitrary squarefree monomial
ideal after interpreting it as the edge ideal of a hypergraph. In this setting, the relevant graph
invariants are replaced by their hypergraph analogues, which we record in the following notation.
\end{remark}

\begin{notation}\label{not:3.1}
Let $\mathcal{H}=(V(\mathcal{H}),E(\mathcal{H}))$ be a hypergraph on $n$ vertices with at least one edge. For each $i\in[n]$, set
\[
g_i
:= \bigl|\{W \subseteq V(\mathcal{H}) : |W|=i \text{ and } W \text{ contains no edge of } \mathcal{H}\}\bigr|,
\]
and define
\[
\alpha := \max\{i\in[n] : g_i \ne 0\},
\qquad
\delta := \min\{|E| : E\in E(\mathcal{H})\}.
\]
Let $I(\mathcal{H})$ denote the edge ideal of a hypergraph $\mathcal{H}$ such that
\[
I(\mathcal{H}) := \Bigl(\prod_{x\in E} x : E\in E(\mathcal{H})\Bigr).
\]
\end{notation}

By replacing $I(G)$ with $I(\mathcal{H})$ in \Cref{eq:h_poly_edge} and using \Cref{not:3.1}, one obtains the expression for the $h$-polynomial of any squarefree monomial ideal.

\section{Degrees of $h$-polynomials for cover ideals}\label{sec2:cover_ideals}

In this section, we give explicit formulas for the $h$-polynomial of cover ideals. In the spirit of
\Cref{rem:gen_h_poly}, we then extend these formulas to the Alexander dual of an arbitrary squarefree
monomial ideal. Finally, using that the $a$-invariant is preserved under Alexander duality for
squarefree monomial ideals, we relate the degrees of the $h$-polynomials of edge  and cover ideals.

We begin by recalling the standard-monomial description of Stanley--Reisner
rings (see \cite[Proposition 1.5.1]{herzog2011monomial}).

\begin{prop}\label{prop:stdmon}
The set of monomials $x_1^{a_1}\cdots x_n^{a_n}$ such that 
\[
\supp(a):=\{i : a_i\neq 0\} \in \Delta(G)^\vee
\]
forms a $k$-basis of $R/I_{\Delta(G)^\vee}=R/J(G)$.
\end{prop}

\begin{obs}\label{obs:faces}
Let $F\subseteq [n]$. Then $F\in \Delta(G)^\vee$ if and only if $[n]\setminus F\notin \Delta(G)$.
Equivalently, $F\in \Delta(G)^\vee$ if and only if $[n]\setminus F$ is \emph{not} an independent set of $G$.
\end{obs}

\begin{lem}
Let $G$ be a finite simple graph on $n$ vertices. For all $d\ge 1$,
\[
HF_{R/J(G)}(d)
=
\binom{n+d-1}{n-1}
- n\binom{d-1}{n-2}
- \binom{d-1}{n-1}
-\sum_{j=2}^{n-1} g_j\binom{d-1}{n-j-1}.
\]
\end{lem}

\begin{proof}
By \Cref{prop:stdmon}, a degree-$d$ standard monomial is determined by its support
$F\in \Delta(G)^\vee$ together with a choice of positive exponents on the variables
indexed by $F$. If $|F|=i$, then the number of degree-$d$ monomials supported
exactly on $F$ is $\binom{d-1}{i-1}$ (the number of compositions of $d$ into $i$
positive parts). Therefore,
\begin{equation}\label{eq:HF1-rev}
HF_{R/J(G)}(d)
= \sum_{F\in \Delta(G)^\vee} \binom{d-1}{|F|-1}
= \sum_{i=1}^{n} f_{i-1}(\Delta(G)^\vee)\binom{d-1}{i-1},
\end{equation}
where $f_{i-1}(\Delta(G)^\vee)$ denotes the number of faces of $\Delta(G)^\vee$
of cardinality $i$.

To evaluate $f_{i-1}(\Delta(G)^\vee)$, use \Cref{obs:faces}: a subset
$F\subseteq [n]$ with $|F|=i$ lies in $\Delta(G)^\vee$ precisely when its
complement $[n]\setminus F$ (of size $n-i$) is not independent. Hence
\[
f_{i-1}(\Delta(G)^\vee) = \binom{n}{i}-g_{n-i},
\]
and \eqref{eq:HF1-rev} becomes
\begin{equation}\label{eq:HF2-rev}
HF_{R/J(G)}(d)
= \sum_{i=1}^{n}\binom{n}{i}\binom{d-1}{i-1} - \sum_{i=1}^{n} g_{n-i}\binom{d-1}{i-1}.
\end{equation}

We now use the binomial identity
\begin{equation}\label{eq:binom-id}
\binom{n+d-1}{n-1} = 
\sum_{i=1}^{n}\binom{n}{i}\binom{d-1}{i-1},
\end{equation}
which counts the number of weak compositions of $d$ into $n$ parts by grouping
according to the number of nonzero parts.

Finally, note that $g_{n-i}=0$ unless $n-i\le \alpha$, and that the terms
corresponding to $i=0$ or $i=n$ in \eqref{eq:HF2-rev} contribute
$g_n\binom{d-1}{-1}=0$ and $g_0\binom{d-1}{n-1}=\binom{d-1}{n-1}$, respectively.
Also the term $i=n-1$ contributes $g_1\binom{d-1}{n-2}=n\binom{d-1}{n-2}$.
Reindexing the remaining sum via $j=n-i$ yields
\[
\sum_{i=1}^{n} g_{n-i}\binom{d-1}{i-1} 
=
n\binom{d-1}{n-2}+\binom{d-1}{n-1}+\sum_{j=2}^{n-1} g_j\binom{d-1}{n-j-1}.
\]
Substituting this and \eqref{eq:binom-id} into \eqref{eq:HF2-rev} gives the stated formula.
\end{proof}

\begin{remark}\label{rem:splitAB}
For later use, write $HF_{R/J(G)}(d)=A_d-B_d$ for $d\ge 1$, where
\[
A_d:=\binom{n+d-1}{n-1}-n\binom{d-1}{n-2}-\binom{d-1}{n-1},
\qquad
B_d:=\sum_{j=2}^{n-1} g_j\binom{d-1}{n-j-1}.
\]
Then the Hilbert series can be written as
\[
H_{R/J(G)}(t)
=
1+\sum_{d\ge 1}HF_{R/J(G)}(d)t^d
=
\Bigl(1+\sum_{d\ge 1}A_dt^d\Bigr)-\sum_{d\ge 1}B_dt^d.
\]
\end{remark}

We now provide an expression for the $h$-polynomial of cover ideals.

\begin{thm}\label{lem:HS1}
Let $G$ be a finite simple graph on $n$ vertices with $\alpha=\alpha(G)$.
Then the $h$-polynomial of $R/J(G)$ is
\[
h_{R/J(G)}(t)
=
\sum_{k=0}^{n-\alpha-1} (k+1)t^k
+
\sum_{s=-1}^{\alpha-3} E_{s+3} t^{n-s-3},
\]
where, for $-1\le s\le \alpha-3$,
\[
E_{s+3}
=
(n-s-2) -
\sum_{j=s+3}^{\alpha} \binom{j-2}{s+1}(-1)^{j-s-1} g_j.
\]
\end{thm}

\begin{proof}
We simplify the two pieces appearing in \Cref{rem:splitAB}.

Using the standard generating functions
\[
\sum_{d\ge 0}\binom{n+d-1}{n-1}t^d=\frac{1}{(1-t)^n},
\qquad
\sum_{d\ge 0}\binom{d}{r}t^d=\frac{t^r}{(1-t)^{r+1}},
\]
one obtains
\begin{align*}
1+\sum_{d\ge 1}A_dt^d
&=
\sum_{d\ge 0}\binom{n+d-1}{n-1}t^d
- n\sum_{d\ge 1}\binom{d-1}{n-2}t^d
- \sum_{d\ge 1}\binom{d-1}{n-1}t^d \\
&=
\frac{1}{(1-t)^n}
- n\cdot \frac{t^{n-1}}{(1-t)^{n-1}}
- \frac{t^{n}}{(1-t)^{n}} \\
&=
\frac{1-t^n}{(1-t)^n}-\frac{nt^{n-1}}{(1-t)^{n-1}}\\
&=
\frac{1+t+\cdots+t^{n-1}-nt^{n-1}}{(1-t)^{n-1}}\\
&=
\frac{1}{(1-t)^{n-1}}\sum_{i=0}^{n-1}(t^i-t^{n-1})\\
&= 
\frac{1+2t+3t^2+\cdots+(n-1)t^{n-2}}{(1-t)^{n-2}}.
\end{align*}

Since $g_j=0$ for $j>\alpha$, we may sum only up to $\alpha$:
\begin{align*}
\sum_{d\ge 1}B_dt^d
&=
\sum_{d\ge 1}\Bigg(\sum_{j=2}^{\alpha} g_j\binom{d-1}{n-j-1}\Bigg)t^d
=
\sum_{j=2}^{\alpha} g_j \sum_{d\ge 1}\binom{d-1}{n-j-1}t^d \\
&=
\sum_{j=2}^{\alpha} g_j \cdot \frac{t^{n-j}}{(1-t)^{n-j}}.
\end{align*}
Factoring out $(1-t)^{-(n-2)}$, expanding $(1-t)^{j-2}$, and switching the order of summation gives
\begin{align*}
\sum_{d\ge 1}B_dt^d
&=
\frac{1}{(1-t)^{n-2}}
\sum_{j=2}^{\alpha} g_j t^{n-j}(1-t)^{j-2} \\
&=
\frac{1}{(1-t)^{n-2}}
\sum_{j=2}^{\alpha} g_j t^{n-j}\sum_{m=0}^{j-2}\binom{j-2}{m}(-1)^m t^m \\
&=
\frac{1}{(1-t)^{n-2}}
\sum_{k=n-\alpha}^{n-2}\Bigg(\sum_{j=n-k}^{\alpha}\binom{j-2}{k-n+j}(-1)^{k-n+j}g_j\Bigg)t^k.
\end{align*}

Combining contributions from $A_d$ and $B_d$  in \Cref{rem:splitAB} yields
\begin{align*}
H_{R/J(G)}(t)
&=
\frac{1}{(1-t)^{n-2}}
\Biggl(
\sum_{k=0}^{n-2}(k+1)t^k
-
\sum_{k=n-\alpha}^{n-2}\Bigg(\sum_{j=n-k}^{\alpha}\binom{j-2}{k-n+j}(-1)^{k-n+j}g_j\Bigg)t^k
\Biggr)\\
&=
\frac{1}{(1-t)^{n-2}}
\Biggl(
\sum_{k=0}^{n-\alpha-1}(k+1)t^k
+
\sum_{k=n-\alpha}^{n-2}\Bigg((k+1)-\sum_{j=n-k}^{\alpha}\binom{j-2}{k-n+j}(-1)^{k-n+j}g_j\Bigg)t^k
\Biggr).
\end{align*}
Now set $k=n-s-3$. Then $k$ ranges from $n-\alpha$ to $n-2$ exactly when
$s$ ranges from $\alpha-3$ down to $-1$, and the coefficient of $t^{n-s-3}$ becomes
\[
(n-s-2)-\sum_{j=s+3}^{\alpha}\binom{j-2}{s+1}(-1)^{j-s-1}g_j
=:E_{s+3}.
\]
This gives
\[
H_{R/J(G)}(t)
=
\frac{1}{(1-t)^{n-2}}
\Biggl(
\sum_{k=0}^{n-\alpha-1}(k+1)t^k
+
\sum_{s=-1}^{\alpha-3}E_{s+3}t^{n-s-3}
\Biggr).
\]
Finally, the first sum has degree at most $n-\alpha-1$, while the second sum
has all terms of degree at least $n-\alpha$, so no cancellation occurs in the
numerator. Hence the displayed numerator is the $h$-polynomial of $R/J(G)$.
\end{proof}

\begin{remark}
    As in \Cref{rem:gen_h_poly}, one can generalize the expression of the $h$-polynomial of cover ideals from \Cref{lem:HS1} to Alexander dual of any squarefree monomial ideal (using Notation~\ref{not:3.1}) as follows: 
\begin{align*}
h_{R/I(\mathcal{H})^{\vee}}(t)
&=
\sum_{k=0}^{n-\alpha-1}\binom{k+\delta-1}{\delta-1}t^k +
\sum_{k=n-\alpha}^{n-\delta}\left(
\binom{k+\delta-1}{\delta-1}
-
\sum_{j=n-k}^{\alpha} \binom{j-\delta}{k-n+j}(-1)^{k-n+j} g_j 
\right)t^k 
\end{align*}
where $\dim R/I(\mathcal{H})^{\vee} = n-\delta$.
\end{remark}

 As in the case of edge ideals, the invariant  $g(G)$ governs when $\deg h_{R/J(G)}(t)$ achieves its maximal possible value $n-2$.

\begin{cor}\label{cor:degree}
Let $G$ be a finite simple graph on $n$ vertices with independence number $\alpha$.
\begin{enumerate}
\item[(a)] $n-\alpha-1\le \deg h_{R/J(G)}(t)\le n-2$.
\item[(b)] $\deg h_{R/J(G)}(t)=n-2$ if and only if $g(G)\neq 1$.
\item[(c)] If $g(G)=1$, then $E_2=0$ and:
\begin{enumerate}
\item[(i)] if $E_3=\cdots=E_{\alpha}=0$, then $\deg h_{R/J(G)}(t)=n-\alpha-1$;
\item[(ii)] otherwise,
\[
\deg h_{R/J(G)}(t)=n-d-3,
\qquad
d=\min\{s: E_{s+3}\neq 0\},
\]
where $0\le s\le \alpha-3$.
\end{enumerate}
\end{enumerate}
\end{cor}

\begin{proof}
(a) This is immediate from \Cref{lem:HS1}.

(b) By \Cref{lem:HS1}, we have $\deg h_{R/J(G)}(t)=n-2$ if and only if the
coefficient of $t^{n-2}$ is nonzero, i.e., $E_2\neq 0$.
Plugging $s=-1$ into the formula for $E_{s+3}$ gives
\[
E_2
=
n-1-\sum_{j=2}^{\alpha}(-1)^j g_j.
\]
Since $g_1=n$, this can be rewritten as
\[
E_2
=
-1+\sum_{j=1}^{\alpha}(-1)^{j-1}g_j
=
-1+g(G).
\]
Hence $E_2\neq 0$ if and only if $g(G)\neq 1$.

(c) If $g(G)=1$, then $E_2=0$ by (b), so $\deg h_{R/J(G)}(t)<n-2$. The remaining
statements follow by inspecting the explicit form of $h_{R/J(G)}(t)$ in
\Cref{lem:HS1}.
\end{proof}

Part (c)(ii) is analogous to the corresponding statement for edge ideals in
\cite[Corollary 3.4]{biermann2024degree}.

We conclude this section by relating the degrees of the $h$-polynomials of edge and cover ideals.

\begin{thm}\label{thm:deg_edge_cover}
Let $G$ be a finite simple graph.
Then there exists an integer $d$ with $0\le d\le \alpha-1$ such that
\[
\deg h_{R/J(G)}(t)=n-2-d
\quad\Longleftrightarrow\quad
\deg h_{R/I(G)}(t)=\alpha-d.
\]   
\end{thm}
\begin{proof}
    It follows from \cite[Theorem 4.36]{miller_thesis} that 
    \[
    \mathfrak{a}(R/I(G))=\mathfrak{a}(R/J(G)).
    \]
Since $\dim(R/I(G))=\alpha$ and $\dim(R/J(G))=n-2$, we have
\[
\deg h_{R/I(G)}(t)-\alpha=\deg h_{R/J(G)}(t)-(n-2).
\]
 Thus, the statement holds.
\end{proof}

\section{Degree of the $h$-polynomial via the independence polynomial and $M(G)$}
\label{sec:deg-h-roots-independence}

A guiding theme of this section and the next is that the degrees of the $h$-polynomials of $R/I(G)$ and
$R/J(G)$ can be studied through the independence polynomial $P_G(x)$.  We first express the
$h$-polynomial of $R/I(G)$ in terms of $P_G(x)$, and then derive degree formulas for both edge and
cover ideals in terms of the invariant $M(G)$, the multiplicity of $x=-1$ as a root of $P_G(x)$.

We also clarify the relationship between the present approach and \cite{biermann2024degree}.  In
\cite{biermann2024degree}, the degree $\deg h_{R/I(G)}(t)$ is governed by the invariant $g(G)$ and a
collection of alternating sums $D_s$ (see \Cref{thm:h_poly_edge}).  We show that these same quantities
are controlled by the behavior of $P_G(x)$ at $x=-1$: the invariant $g(G)$ detects whether $x=-1$ is a
root of $P_G(x)$, while the first nonzero $D_s$ determines the multiplicity of that root.  Moreover, we
interpret each $D_s$ as the value of an appropriate derivative of $P_G(x)$ at $x=-1$.

To make this precise, we adopt the following notation for root multiplicities.

\begin{notation}\label{not:ord}
For a polynomial $f(x)$ and a scalar $a$, we write $\ord_{x=a} f(x)$ for the multiplicity of $x=a$
as a root of $f(x)$.
\end{notation}

With this notation in place, we record the multiplicity of the distinguished root $x=-1$ as a graph invariant.

\begin{defn}\label{def:M}
Let $G$ be a graph with independence polynomial $P_G(x)$. Define
\[
M(G) := \ord_{x=-1} P_G(x),
\]
the multiplicity of $-1$ as a root of $P_G(x)$.
\end{defn}

Since $M(G)$ is the multiplicity of a root of $P_G(x)$, we have
$M(G)\le \deg P_G(x)=\alpha$. If $G$ has at least one edge, this bound can be improved.

\begin{lem}\label{lem:M_bound}
If $G$ has at least one edge, then
\[
M(G)\ \le\ \alpha-1.
\]
\end{lem}

\begin{proof}
Assume $G$ has at least one edge. If $M(G)\ge \alpha$, then $(1+x)^{\alpha}$ divides $P_G(x)$.
Since $\deg P_G(x)=\alpha$, we get $P_G(x)=c(1+x)^{\alpha}$ for some constant $c\in\Bbbk$.
Because $P_G(0)=1$, we must have $c=1$, so $P_G(x)=(1+x)^{\alpha}$.
Comparing the coefficients of $x$ gives $n=\alpha$, which forces $G$ to have no edges, a contradiction.
\end{proof}

We now state the main result of this section, which relates $M(G)$ to the degrees of the
$h$-polynomials of $R/I(G)$ and $R/J(G)$.

\begin{thm}\label{lem:deg-via-ord}
We have
\begin{align*}
 \deg h_{R/I(G)}(t)&=\alpha-M(G),\\
 \deg h_{R/J(G)}(t)&=n-2-M(G).   
\end{align*}

In particular, $\deg h_{R/I(G)}(t)=\alpha$ if and only if $P_G(-1)\neq 0$.
\end{thm}

\begin{proof}
Setting $u=\frac{1}{1-t}$ in \eqref{eq:h-from-independent-poly} yields 
\[
h_{R/I(G)}(t)=(1-t)^{\alpha}P_G(u-1)=u^{-\alpha}P_G(u-1).
\]
Let $M=M(G)$. So we may write
\[
P_G(x)=(x+1)^M Q(x)
\]
with $Q(-1)\neq 0$. Substituting $x=u-1$ gives
\[
h_{R/I(G)}(t)=u^{M-\alpha}Q(u-1).
\]
Note that $Q(u-1)$ is a polynomial in $u$ of degree $\deg Q=\alpha-M$ which has nonzero constant term
$Q(-1)$. Furthermore, the term $Q(-1)u^{M-\alpha}=Q(-1)(1-t)^{\alpha-M}$ is the largest degree term of $h_{R/I(G)}(t)$ as a polynomial in $(1-t)$, and therefore also as a polynomial in $t$. Hence $\deg h_{R/I(G)}(t)=\alpha-M$.

Lastly, the expression for the degree of the $h$-polynomial of $R/J(G)$ follows from  \Cref{thm:deg_edge_cover}.
\end{proof}

Using \Cref{lem:deg-via-ord}, we can translate the degree formulas for the \(h\)-polynomials into a relationship between \(M(G)\) and the \(\mathfrak{a}\)-invariants of \(R/I(G)\) and \(R/J(G)\).

\begin{remark}
    Since $\mathfrak{a}$-invariants of $R/I(G)$ and $R/J(G)$ coincide, it follows from \Cref{lem:deg-via-ord} that 
    \[
    \mathfrak{a}(R/J(G))=\mathfrak{a}(R/I(G))= -M(G).
    \]
\end{remark}

Recall from  \Cref{thm:h_poly_edge} that  $\deg h_{R/I(G)}(t)$ can be read off from the first nonzero $D_s$,  for $s\ge 0$, where
\begin{equation}\label{eq:Ds-def}
D_s=\sum_{j=s+1}^{\alpha}(-1)^{j-1-s}\binom{j}{s+1}g_j.
\end{equation}
On the other hand, \Cref{lem:deg-via-ord} shows that the same
degree is determined by the invariant $M(G)$. The next lemma makes the link between these two
descriptions explicit by expressing each $D_s$ in terms of the derivatives of the independence polynomial
evaluated at $x=-1$, thereby connecting the $M(G)$ approach directly to \cite{biermann2024degree}.

\begin{lem}\label{lem:Ds-derivative}
For $s\ge 0$, we have
\[
D_s=\frac{1}{(s+1)!}P_G^{(s+1)}(-1).
\]
\end{lem}

\begin{proof}
Differentiate the independence polynomial $(s+1)$ times:
\[
P_G^{(s+1)}(x)=\sum_{j=s+1}^{\alpha} g_jj(j-1)\cdots (j-s)x^{j-(s+1)}
=(s+1)!\sum_{j=s+1}^{\alpha} g_j\binom{j}{s+1}x^{j-(s+1)}.
\]
Evaluating at $x=-1$ yields
\[
\frac{1}{(s+1)!}P_G^{(s+1)}(-1)
=\sum_{j=s+1}^{\alpha}(-1)^{j-1-s}\binom{j}{s+1}g_j
=D_s.
\]
\end{proof}

Thus the sequence $\{D_s\}_{s\ge 0}$ encodes the order of vanishing of $P_G(x)$ at $x=-1$, and hence the invariant $M(G)$.

\begin{cor}\label{cor:M-via-Ds}
Let $M=M(G)$. Then
\[
P_G(-1)=P_G'(-1)=\cdots=P_G^{(M-1)}(-1)=0,\quad P_G^{(M)}(-1)\neq 0,
\]
equivalently,
\[
D_0=D_1=\cdots=D_{M-2}=0,\quad D_{M-1}\neq 0.
\]
In particular, the first nonzero $D_s$ records the order of vanishing of $P_G$ at $-1$.
\end{cor}

The following result recovers the expression of the $h$-polynomial of edge ideals from \cite{biermann2024degree}.

\begin{prop}\label{prop:h-derivative-expansion}
We have
\[
h_{R/I(G)}(t)=P_G(-1)(1-t)^{\alpha}+\sum_{s=0}^{\alpha-1} D_s(1-t)^{\alpha-s-1}.
\]
\end{prop}

\begin{proof}
As in the proof of \Cref{lem:deg-via-ord}, let $u=\frac{1}{1-t}$ so that
$h_{R/I(G)}(t)=u^{-\alpha}P_G(u-1)$. Expand $P_G$ at $-1$:
\[
P_G(u-1)=\sum_{k=0}^{\alpha}\frac{P_G^{(k)}(-1)}{k!}u^k.
\]
Hence
\[
h_{R/I(G)}(t)=\sum_{k=0}^{\alpha}\frac{P_G^{(k)}(-1)}{k!}u^{k-\alpha}
=\sum_{k=0}^{\alpha}\frac{P_G^{(k)}(-1)}{k!}(1-t)^{\alpha-k}.
\]
The stated formula follows from \Cref{lem:Ds-derivative} (with $k=s+1$).
\end{proof}

\begin{remark}
\Cref{lem:deg-via-ord} is also immediate from \Cref{prop:h-derivative-expansion}:
the first nonzero derivative at $-1$ occurs at order $M(G)$. So, the largest power of $t$
comes from the term $(1-t)^{\alpha-M(G)}$.
\end{remark}

Since $M(G)$ controls $\deg h_{R/I(G)}(t)$ by \Cref{lem:deg-via-ord}, it is useful to understand how
$M(G)$ changes under basic graph operations.

\begin{question}\label{q:operations}
Under which graph operations does the multiplicity $M(G)$ behave in a predictable~way?
\end{question}

As a first test case, we examine induced subgraphs; the next example shows that $M(G)$ need not vary monotonically.

\begin{ex}\label{ex:paths-nonmonotone}
Let $P_n$ denote the path on $n$ vertices. For $n=2,3,4,5$ the independence polynomials are
\begin{align*}
   P_{P_2}(x)&=1+2x,\\
   P_{P_3}(x)&=1+3x+x^2,\\ 
   P_{P_4}(x)&=1+4x+3x^2,\\
   P_{P_5}(x)&=1+5x+6x^2+x^3.
\end{align*}
Evaluating at $x=-1$ gives
\[
P_{P_2}(-1)=-1,\qquad
P_{P_3}(-1)=-1,\qquad
P_{P_4}(-1)=0,\qquad
P_{P_5}(-1)=1.
\]
Hence $M(P_4)=1$, while $M(P_2)=M(P_3)=M(P_5)=0$.
In particular, although $P_4$ is an induced subgraph of $P_5$, we have $M(P_4)>M(P_5)$; and although $P_3$ is an induced subgraph of $P_4$, we have $M(P_3)<M(P_4)$.
\end{ex}

Thus $M(G)$ can increase, decrease, or remain unchanged when passing to induced subgraphs. This motivates a search for graph operations under which $M(G)$ behaves more systematically,
and we hope it prompts further study of $M(G)$ in its own right.

\section{Trees of radius at most two: $(\reg,\deg h)$ for cover ideals}
\label{sec:rad2-trees}

In this section we study trees of radius at most $2$. Throughout, let $T$ be a tree of radius at most 2 and $R_T=\Bbbk[x_v:v\in V(T)]$. We  first compute $\pdim(R_T/I(T))$, and hence $\reg(R_T/J(T))$ by Alexander duality. Then, we determine the degree of the $h$-polynomial of $R_T/J(T)$ via the multiplicity
$M(T)$ from \Cref{sec:deg-h-roots-independence}.

We begin by setting our notation.

\begin{notation}\label{not:radius2}
Let $T$ be a tree of radius at most $2$. In other words, there is a vertex $c$ such that
$\dist_T(c,v)\le 2$ for all $v\in V(T)$. Fix such a vertex $c$ and set $B=\deg_T(c)$. 
Write the neighbor set of $c$ as a disjoint union
\[
N_T(c)=\{\ell_1,\dots,\ell_{L_1}\}\sqcup\{v_1,\dots,v_m\},
\]
where $\ell_1,\dots,\ell_{L_1}$ are leaves adjacent to $c$, and each $v_i$ has
at least one leaf away from $c$ (equivalently, $\deg_T(v_i)\ge 2$). Thus $m=B-L_1$.
For each $i$, set
\[
t_i  =  \bigl|N_T(v_i)\setminus\{c\}\bigr| \ \ge 1,
\qquad\text{and}\qquad
L_2  =  \sum_{i=1}^m t_i.
\]
Then $L_2$ is the number of vertices at distance $2$ from $c$ (all of which are leaves).
Note that $T$ has radius $2$ if and only if $L_2>0$ (equivalently $m>0$), and radius $1$
if and only if $L_2=0$ (so $T\cong K_{1,B}$).

Using this notation, we have \(|V(T)|=n=1+B+L_2\). Set \(p(T)\ :=\ \pdim(R_T/I(T))\).
\end{notation}

\begin{figure}[ht]
    \centering
\includegraphics[width=0.25\linewidth]{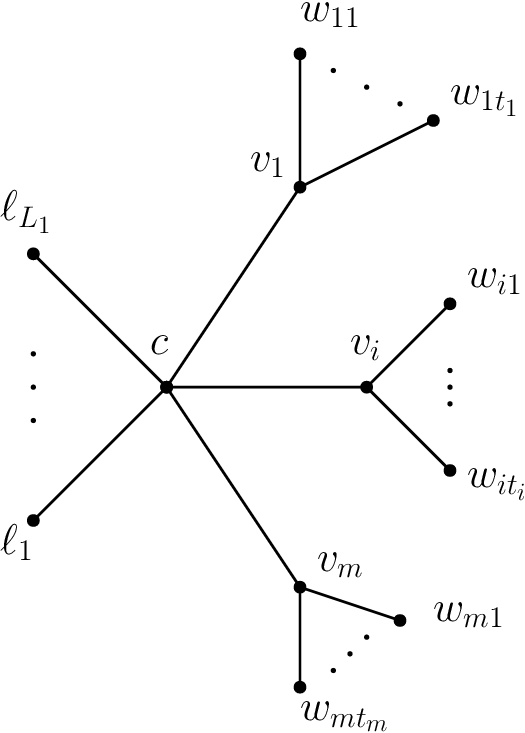}
    \caption{A tree of radius at most $2$ with parameters as in \Cref{not:radius2}.}
    \label{fig:gen_tree_rad_2}
\end{figure}

\subsection{Projective dimension and regularity}

The following recursive result of Jacques and Katzman is our main tool.

\begin{thm}{\cite[Theorem 4.8]{jacques2005betti}}\label{thm:Jacques_Katzman}
Let $F$ be a forest containing a vertex of degree at least two. Then there exists a vertex
$v$ with neighbors
\[
N_F(v)=\{v_1,\dots,v_n\}, \qquad n\ge 2,
\]
such that $v_1,\dots,v_{n-1}$ are leaves. Define
\[
F' = F\setminus\{v_1\}, \qquad
F'' = F\setminus\{v,v_1,\dots,v_n\}.
\]
Then
\[
\pdim(R_F/I(F))  =  \max\bigl\{ \pdim(R_{F'}/I(F')),\ \pdim(R_{F''}/I(F''))+n\bigr\},
\]
where $R_H=\Bbbk[x_u:u\in V(H)]$ for any forest $H$.
\end{thm}

\begin{remark}\label{rem:isolated-vertices-pdim}
The forest $F''$ in \Cref{thm:Jacques_Katzman} may contain isolated vertices. These do not affect
projective dimension; they contribute extra variables but no generators to the edge ideal.
\end{remark}

\begin{thm}\label{prop:radius2-pdim}
Let $T$ be as in \Cref{not:radius2}. Then
\[
p(T) = \max\{B,\ L_2+1\}.
\]
In particular, \(\reg(R_T/J(T))=\max\{B-1,\ L_2\}\).

\end{thm}

\begin{proof}
We induct on $L_2\ge 0$.

If $L_2=0$, then $T\cong K_{1,B}$ (a star) and $p(T)=B$. So, the formula holds for the base case.

Assume $L_2\ge 1$ and the formula holds for all trees of radius at most $2$ with strictly fewer
distance--$2$ leaves. Choose an index $i$. Then $v_i$ has neighbors
\[
N_T(v_i)=\{c,w_{i1},\dots,w_{i t_i}\},
\]
where $w_{i1},\dots,w_{i t_i}$ are leaves and $\deg_T(v_i)=t_i+1$.
Apply \Cref{thm:Jacques_Katzman} with $v=v_i$, $v_1=w_{i1}$, and the remaining neighbors
$v_2,\dots,v_n$ taken to be $w_{i2},\dots,w_{i t_i},c$ (so $n=t_i+1$). Then
\[
T' = T\setminus\{w_{i1}\},\qquad
T'' = T\setminus\{v_i,c,w_{i1},\dots,w_{i t_i}\}.
\]

In $T'$, the vertex $c$ is still a center of degree $B$ and $L_2$ drops by $1$. Hence, by induction,
\[
p(T')=\max\{B,(L_2-1)+1\}=\max\{B,L_2\}.
\]

In $T''$, removing $c$ makes each $\ell_j$ isolated, and for each $j\ne i$ the vertex $v_j$ becomes the
center of a star $K_{1,t_j}$. Thus $T''$ is a disjoint union of stars together with isolated vertices, and
by the additivity of projective dimension under disjoint union (and \Cref{rem:isolated-vertices-pdim}),
\[
p(T'')=\sum_{j\ne i} p(K_{1,t_j})=\sum_{j\ne i} t_j = L_2-t_i.
\]
Therefore \Cref{thm:Jacques_Katzman} yields
\begin{align*}
 p(T)& =\max\{p(T'),\ p(T'')+\deg_T(v_i)\}\\ 
& =\max\bigl\{\max\{B,L_2\},\ (L_2-t_i)+(t_i+1)\bigr\}\\
&=\max\{B,L_2+1\},   
\end{align*}
as claimed.

The regularity statement follows from \Cref{prop:radius2-pdim}. 
\end{proof}

\subsection{Degree of the $h$-polynomial}

We next compute $P_T(x)$ and $M(T)$, and then apply
\Cref{lem:deg-via-ord} to obtain the degree.

\begin{prop}\label{prop:PT}
Let $T$ be as in \Cref{not:radius2}. Suppose $L_2>0$. Then
\[
P_T(x)= (1+x)^{L_1}\prod_{i=1}^m\Bigl(x+(1+x)^{t_i}\Bigr) + x(1+x)^{L_2}.
\]
\end{prop}
\begin{proof}
We count independent sets in two disjoint cases.

Suppose $c$ is not chosen.
We may choose any subset of the $L_1$ leaves $\ell_1,\dots,\ell_{L_1}$, contributing $(1+x)^{L_1}$.
Independently, for each branch $v_i$ we either choose $v_i$ (contributing $x$) or do not choose $v_i$
and then choose any subset of its $t_i$ leaves (contributing $(1+x)^{t_i}$). Thus branch $i$ contributes
$x+(1+x)^{t_i}$. Multiplying gives
\[
(1+x)^{L_1}\prod_{i=1}^m\bigl(x+(1+x)^{t_i}\bigr).
\]
Suppose $c$ is chosen.
This contributes $x$, forces all neighbors of $c$ to be absent, and allows any subset of the $L_2$
distance--$2$ leaves, contributing $(1+x)^{L_2}$. Thus this case contributes $x(1+x)^{L_2}$.

\smallskip
Adding the two cases yields the stated formula.
\end{proof}

\begin{remark}\label{lem:alpha-rad2}
Let $T$ be as in \Cref{not:radius2}. Then
\[
\alpha(T)=
\begin{cases}
B, & \text{if } L_2=0,\\[2pt]
L_2+1, & \text{if } L_2>0 \text{ and } L_1=0,\\[2pt]
L_1+L_2, & \text{if } L_2>0 \text{ and } L_1>0.
\end{cases}
\]
\end{remark}

One can express $M(T)$ using the parameters of $T$.

\begin{lem}\label{lem:MT-rad2}
Let $T$ be as in \Cref{not:radius2}. Then the multiplicity $M(T)$ is
\[
M(T)=
\begin{cases}
0, & \text{if } L_2=0,\\[2pt]
0, & \text{if } L_2>0 \text{ and } L_1=0,\\[2pt]
\min\{L_1,L_2\}, & \text{if } L_1>0 \text{ and } L_1\neq L_2,\\[2pt]
L_1, & \text{if } L_1=L_2 \text{ and } m \text{ is odd},\\[2pt]
L_1+1, & \text{if } L_1=L_2 \text{ and } m \text{ is even}.
\end{cases}
\]
\end{lem}

\begin{proof}
If $L_2=0$, then $T\cong K_{1,B}$ and $P_T(x)=(1+x)^B+x$. Hence $P_T(-1)=-1\neq 0$ and $M(T)=0$.

Assume $L_2>0$. If $L_1=0$, then \Cref{prop:PT} gives
\[
P_T(-1)=\prod_{i=1}^m\bigl(-1+(1-1)^{t_i}\bigr)=(-1)^m\neq 0,
\]
so again $M(T)=0$.

Now assume $L_1>0$ and set $L=\min\{L_1,L_2\}$. By \Cref{prop:PT},
\[
P_T(x)=(1+x)^L Q(x),
\]
where
\[
Q(x)=(1+x)^{L_1-L}\prod_{i=1}^m\bigl(x+(1+x)^{t_i}\bigr)+x(1+x)^{L_2-L}.
\]
If $L_1\neq L_2$, then $Q(-1)\neq 0$:
\begin{itemize}
\item If $L_1<L_2$, then $Q(-1)=\prod_{i=1}^m(-1)+(-1)\cdot 0 = (-1)^m\neq 0$.
\item If $L_2<L_1$, then $Q(-1)=0\cdot\prod_{i=1}^m(-1)+(-1)=-1\neq 0$.
\end{itemize}
Thus $M(T)=L=\min\{L_1,L_2\}$.

Finally, assume $L_1=L_2=U$. Then
\[
P_T(x)=(1+x)^U Q(x)
\qquad\text{with}\qquad
Q(x)=\Big(\prod_{i=1}^m f_i(x)\Big)+x,\quad f_i(x)=x+(1+x)^{t_i}.
\]
Since $f_i(-1)=-1$ for all $i$, we have
\[
Q(-1)=(-1)^m-1=
\begin{cases}
-2, &\text{if } m \text{ odd},\\
0, & \text{if } m \text{ even}.
\end{cases}
\]
If $m$ is odd then $Q(-1)\neq 0$ and $M(T)=U$.

If $m$ is even, then $Q(-1)=0$. Differentiating,
\[
Q'(x)=\sum_{i=1}^m f_i'(x)\prod_{j\neq i} f_j(x) + 1,
\]
so using $f_j(-1)=-1$ gives
\[
Q'(-1)=(-1)^{m-1}\sum_{i=1}^m f_i'(-1)+1.
\]
Moreover $f_i'(x)=1+t_i(1+x)^{t_i-1}$. Hence $f_i'(-1)\in\{1,2\}$ and in particular
$\sum_i f_i'(-1)\ge m$. Since $m$ is even, $(-1)^{m-1}=-1$, and therefore
\[
Q'(-1)=1-\sum_{i=1}^m f_i'(-1)\le 1-m\neq 0.
\]
Thus $-1$ is a simple root of $Q$. So $M(T)=U+1=L_1+1$.
\end{proof}

Now, we are ready to express the degree of the $h$-polynomial of cover ideals.

\begin{cor}\label{cor:rad2_cover}
Let $T$ be as in \Cref{not:radius2}. Then
\[
\deg h_{R_T/J(T)}(t)=
\begin{cases}
B-1, & \text{if } L_2=0,\\[2pt]
B+L_2-1, & \text{if } L_1=0,\\[2pt]
m+\max\{L_1,L_2\}-1, & \text{if } L_1>0 \text{ and } L_1\neq L_2,\\[2pt]
B-1, & \text{if } L_1=L_2 \text{ and } m \text{ is odd},\\[2pt]
B-2, & \text{if } L_1=L_2 \text{ and } m \text{ is even}.
\end{cases}
\]
\end{cor}

\begin{proof}
By \Cref{thm:deg_edge_cover} and \Cref{lem:deg-via-ord}, we have
\begin{equation}\label{eq:deg-cover-via-M}
\deg h_{R_T/J(T)}(t)=(B+L_2-1)-M(T).
\end{equation}
The displayed piecewise expression is obtained by substituting $B=L_1+m$ and the corresponding
values of $M(T)$ from \Cref{lem:MT-rad2} in \Cref{eq:deg-cover-via-M}.
\end{proof}

\subsection{Realizable $(\reg,\deg h)$ pairs}
Let $(r,d)=\bigl(\reg(R_T/J(T)),\ \deg h_{R_T/J(T)}(t)\bigr)$. 
We now describe the set of realizable pairs
\[
\mathcal{T}_n
=\Bigl\{(r,d)\ :\ T \text{ is a tree of radius at most $2$ on } n \text{ vertices}\Bigr\}.
\]
\begin{thm}\label{cor:radius2-realizable-pairs}
Fix $n\ge 4$ and define
\[
\mathcal{A}_n=\Bigl\{(r,d)\in\mathbb{Z}^2:\ 
\Bigl\lceil\frac{n-2}{2}\Bigr\rceil\le r\le n-2,\ 
r\le d\le \min\{n-2,\ 2r-1\}\Bigr\}.
\]
If $n$ is odd, also set
\[
\mathcal{B}_n
=\Bigl\{(r,r-1):\ 
\Bigl\lceil\frac{n-2}{2}\Bigr\rceil\le r\le 
\Bigl\lfloor\frac{2n-5}{3}\Bigr\rfloor
\Bigr\}.
\]
Then:
\begin{enumerate}
\item[(i)] If $n$ is even, then $\mathcal{T}_n=\mathcal{A}_n$.
\item[(ii)] If $n$ is odd, then $\mathcal{T}_n=\mathcal{A}_n \cup \mathcal{B}_n$.
\end{enumerate}
\end{thm}

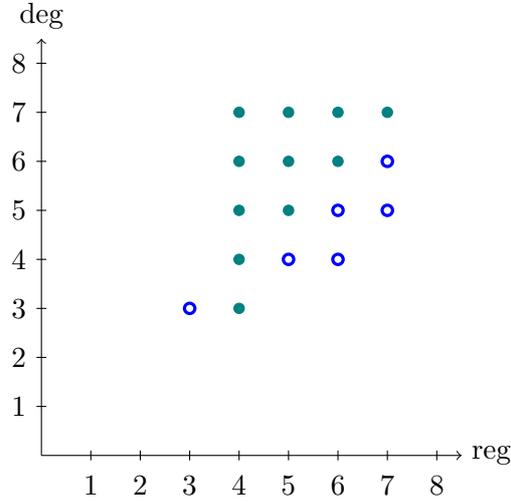
\begin{figure}[ht]
\centering
\begin{tikzpicture}[scale=0.65]

    \draw[->] (0,0) -- (8.5,0) node[right] {$\text{reg}$};
    \draw[->] (0,0) -- (0,8.5) node[above] {$\text{deg}$};

    \foreach \x in {1,2,3,4,5,6,7,8} {
        \draw (\x,-0.1) -- (\x,0.1);
        \node[below] at (\x,-0.2) {\x};
    }
    
    \foreach \y in {1,2,3,4,5,6,7,8} {
        \draw (-0.1,\y) -- (0.1,\y);
        \node[left] at (-0.1,\y) {\y};
    }

    \foreach \x/\y in {
        4/3,4/4,4/5,4/6,4/7,
        5/5,5/6,5/7,
        6/6,6/7,
        7/7
    }{
        \filldraw[teal] (\x,\y) circle (3pt);
    }

    \foreach \x/\y in {3/3,5/4,6/4,6/5,7/5,7/6}{
        \draw[blue, very thick, fill=white] (\x,\y) circle (3pt);
    }
\end{tikzpicture}
\caption{All possible $(\reg,\deg h)$ values for connected graphs on $n=9$ vertices.
Filled points are realized by radius--$\leq 2$ trees; hollow points are not.}
\label{fig:reg_deg_9_holes}
\end{figure}

\begin{proof}
Let $T$ be a tree on $n$ vertices of radius at most $2$ and use \Cref{not:radius2}.
If $L_2=0$, then $T\cong K_{1,n-1}$ and $(r,d)=(n-2,n-2)\in\mathcal{A}_n$.
Assume $L_2>0$. By \Cref{prop:radius2-pdim} we have $r=\max\{B-1,L_2\}$. In addition, recall that $n-2=(B-1)+L_2$. Hence
\[
\Bigl\lceil\frac{n-2}{2}\Bigr\rceil
=\Bigl\lceil\frac{(B-1)+L_2}{2}\Bigr\rceil
\le \max\{B-1,L_2\}=r
\le n-3.
\]
Moreover $d\le n-2$ always by the explicit formulas in \Cref{cor:rad2_cover}. We next show $d\le 2r-1$. Since $L_2=\sum_{i=1}^m t_i$ with each
$t_i\ge 1$, we have $m\le L_2$.

If $r=L_2$, then \Cref{cor:rad2_cover} gives $d\in\{n-2,\ m+L_2-1,\ r\}$ depending on the cases.
In each instance $d\le m+L_2-1\le 2L_2-1=2r-1$.

If $r=B-1$, then \Cref{cor:rad2_cover} gives $d\in\{r,\ r-1,\ m+L_2-1\}$ and
$m+L_2-1\le 2r-1$ because $m\le L_2\le r$.

Thus $(r,d)\in\mathcal{A}_n$ unless possibly $d=r-1$. By \Cref{cor:rad2_cover}, the inequality
$d<r$ occurs exactly in the case $L_1=L_2$ and $m$ even. In that situation,
since $n=1+m+2L_1$,  $n$ is odd. Furthermore $L_2\ge m$ yields the additional
constraint $r\le \lfloor(2n-5)/3\rfloor$ (equivalently $(r,r-1)\in\mathcal{B}_n$).
Hence $\mathcal{T}_n\subseteq \mathcal{A}_n$ when $n$ is even and
$\mathcal{T}_n\subseteq \mathcal{A}_n\cup\mathcal{B}_n$ when $n$ is odd.

\emph{Realizing $\mathcal{A}_n$.}
Let $(r,d)\in\mathcal{A}_n$.
If $r=n-2$ then $d=n-2$ and $K_{1,n-1}$ realizes $(n-2,n-2)$.
Assume $r\le n-3$ and define
\[
L_2:=r,\qquad
L_1:=n-2-d,\qquad
B:=n-1-r,\qquad
m:=B-L_1=d-r+1.
\]
Then $B\ge 2$ and $m\ge 1$ since $d\ge r$. Also $d\le 2r-1$ implies $t_1:=2r-d\ge 1$.
Set
\[
t_1:=2r-d,\qquad t_2=\cdots=t_m:=1,
\]
so $\sum_{i=1}^m t_i=r=L_2$. Construct a radius--$2$ tree with center $c$ adjacent to
$L_1$ leaves and to $m$ vertices $v_1,\dots,v_m$, where $v_i$ has exactly $t_i$ leaf neighbors.
Then $n=1+B+L_2$ holds, and \Cref{cor:rad2_cover} gives
\[
\reg(R_T/J(T))=\max\{B-1,L_2\}=\max\{n-r-2,r\}=r,
\]
since $r\ge \lceil (n-2)/2\rceil$ implies $n-r-2\le r$. Finally, by \Cref{cor:rad2_cover},
the same tree satisfies $\deg h_{R_T/J(T)}(t)=d$ (in the relevant subcase determined by $L_1$ and $L_2$).
Thus $\mathcal{A}_n\subseteq \mathcal{T}_n$.

\emph{Realizing $\mathcal{B}_n$.}
Let $(r,r-1)\in\mathcal{B}_n$ and assume $n$ is odd. Define
\[
L_2:=n-2-r,\qquad
L_1:=L_2,\qquad
B:=r+1,\qquad
m:=B-L_1=2r-n+3.
\]
Then $m\ge 1$ and, since $r\le\lfloor(2n-5)/3\rfloor$, we have $L_2\ge m$.
Choose
\[
t_1:=L_2-m+1\ge 1,\qquad t_2=\cdots=t_m:=1,
\]
so $\sum_i t_i=L_2$. Construct the corresponding radius--$2$ tree. Since $n=1+m+2L_1$ is odd,
the integer $m$ is even, so \Cref{cor:rad2_cover} yields
\[
\deg h_{R_T/J(T)}(t)=B-2=r-1,
\qquad
\reg(R_T/J(T))=\max\{B-1,L_2\}=r.
\]
Thus $\mathcal{B}_n\subseteq\mathcal{T}_n$, completing the proof.
\end{proof}

\section{A Jacques–Katzman type recursion for $M(G)$ on forests}
\label{sec:deg-h-recursion-forests}

In this section, we develop an analogue of Jacques--Katzman's recursion
\cite[Theorem~4.8]{jacques2005betti} to compute the degree of the $h$-polynomial
associated to the edge (and cover) ideal of a forest. It follows from \Cref{lem:deg-via-ord} that
\[
\deg h_{R_F/I(F)}(t)=\alpha(F)-M(F).
\]
Thus it suffices to compute $\alpha(F)$ and $M(F)$ recursively.

\begin{notation}\label{not:M-and-leading}
Let $G$ be a finite simple graph with independence polynomial $P_G(x)$.
Set $u=x+1$ and consider $P_G(u-1)\in \Bbbk[u]$. Then \(M(G)=\ord_{u=0}P_G(u-1)\).

We also record the first nonzero Taylor coefficient at $u=0$ by
\[
c(G):=[u^{M(G)}]P_G(u-1)\in\Bbbk^\times,
\]
so that
\[
P_G(u-1)=c(G)u^{M(G)}+\text{(terms of degree $>M(G)$)}.
\]
\end{notation}

\begin{lem}\label{lem:ord-sum}
Let $A(u),B(u)\in \Bbbk[u]$ be nonzero and set $\nu(A)=\ord_{u=0}A(u)$.
Then
\[
\nu(A+B)\ge \min\{\nu(A),\nu(B)\},
\]
with equality unless $\nu(A)=\nu(B)$ and the leading coefficients cancel.
More precisely, if $\nu(A)=a$ and $\nu(B)=b$, then:
\begin{enumerate}[i]
\item if $a<b$, then $\nu(A+B)=a$ and $[u^a](A+B)=[u^a]A$;
\item if $b<a$, then $\nu(A+B)=b$ and $[u^b](A+B)=[u^b]B$;
\item if $a=b$, then $\nu(A+B)\ge a$ and $[u^a](A+B)=[u^a]A+[u^a]B$.
\end{enumerate}
\end{lem}

\begin{proof}
Write $A(u)=a_0u^a+\text{(terms of degree $>a$)}$ and
$B(u)=b_0u^b+\text{(terms of degree $>b$)}$ with $a_0,b_0\neq 0$.
If $a<b$ (resp.\ $b<a$), then the term $a_0u^a$ (resp.\ $b_0u^b$) cannot be cancelled by any term
of the other polynomial, so $\nu(A+B)=\min\{a,b\}$ and the leading coefficient is inherited from the
smaller-order summand. If $a=b$, then the coefficient of $u^a$ in $A+B$ is $a_0+b_0$, which is nonzero
exactly when no cancellation occurs.
\end{proof}

\begin{remark}
If $H=G\sqcup K_1$ is obtained from $G$ by adding an isolated vertex, then
$P_H(x)=(1+x)P_G(x)$. Hence $M(H)=M(G)+1$ and $\alpha(H)=\alpha(G)+1$, so
\[
\alpha(H)-M(H)=\alpha(G)-M(G),
\]
and consequently $\deg h_{S_H/I(H)}(t)=\deg h_{S_G/I(G)}(t)$.
In particular, isolated vertices may be ignored when computing $\deg h$ for forests.
\end{remark}

We follow the general setup of Jacques and Katzman \cite{jacques2005betti}, modifying it slightly to suit our purposes.

\begin{notation}\label{not:forest_recursive_deg}
Let $F$ be a forest containing a vertex $v$ of degree $n\ge 2$ with neighbor set
\[
N_F(v)=\{v_1,\dots,v_n\},
\]
where $v_1,\dots,v_{n-1}$ are leaves.
Let $F\setminus U$ denote the induced subgraph on $V(F)\setminus U$, and define
the induced subforests
\[
F'  := F\setminus\{v,v_1,\dots,v_{n-1}\},
\qquad
F'' := F\setminus\{v,v_1,\dots,v_n\}.
\]
\end{notation}

\begin{prop}\label{prop:JK-decomposition}
Let $F$ be as in \Cref{not:forest_recursive_deg}. Then
\begin{equation}\label{eq:JK-local-decomp-x}
P_F(x)=(1+x)^{n-1}P_{F'}(x) + xP_{F''}(x).
\end{equation}
Equivalently, after substituting $x=u-1$, one has
\begin{equation}\label{eq:JK-local-decomp-u}
P_F(u-1)=u^{n-1}P_{F'}(u-1) + (u-1)P_{F''}(u-1).
\end{equation}
\end{prop}

\begin{proof}
Partition the independent sets of $F$ into two disjoint cases.

If $v$ is not chosen, then each leaf $v_1,\dots,v_{n-1}$ may be chosen or not chosen independently,
contributing the factor $(1+x)^{n-1}$. The remaining choices form an independent
set in the induced subforest $F'$, contributing $P_{F'}(x)$.
Thus this case contributes $(1+x)^{n-1}P_{F'}(x)$.

If  $v$ is chosen, this contributes a factor of $x$ and forces that none of $v_1,\dots,v_n$ is chosen.
The remaining choices form an independent set in $F''$, contributing $P_{F''}(x)$.
Thus this case contributes $xP_{F''}(x)$.

\smallskip
Adding the two cases yields \eqref{eq:JK-local-decomp-x}. Substituting $x=u-1$
gives \eqref{eq:JK-local-decomp-u}.
\end{proof}

We next extract recursive formulas for $M(F)$ (and the associated leading coefficient $c(F)$)
from the local decomposition in \Cref{prop:JK-decomposition}.

\begin{thm}\label{thm:JK-recursion-M}
Let $F$ be as in \Cref{not:forest_recursive_deg}. Set
\[
r_1:=(n-1)+M(F'),
\qquad\text{and}\qquad
r_2:=M(F'').
\]
Then:
\begin{enumerate}[a]
\item If $r_1<r_2$, then $M(F)=r_1$ and $c(F)=c(F')$.
\item If $r_2<r_1$, then $M(F)=r_2$ and $c(F)=-c(F'')$.
\item If $r_1=r_2=r$, then $M(F)\ge r$ and
\[
[u^{r}]P_F(u-1)=c(F')-c(F'').
\]
In particular, $M(F)=r$ if and only if $c(F')\neq c(F'')$, while $M(F)\ge r+1$
if and only if $c(F')=c(F'')$.
\end{enumerate}
\end{thm}

\begin{proof}
By \Cref{eq:JK-local-decomp-u},
\[
P_F(u-1)=u^{n-1}P_{F'}(u-1) + (u-1)P_{F''}(u-1).
\]
By \Cref{not:M-and-leading}, the Taylor expansions at $u=0$ begin as
\begin{align*}
P_{F'}(u-1)  &= c(F')u^{M(F')}+\text{(terms of degree $>M(F')$)},\\
P_{F''}(u-1) &= c(F'')u^{M(F'')}+\text{(terms of degree $>M(F'')$)},
\end{align*}
with $c(F'),c(F'')\neq 0$.

Multiplying the first expansion by $u^{n-1}$ shifts degrees up by $n-1$, hence
\[
u^{n-1}P_{F'}(u-1)=c(F')u^{(n-1)+M(F')}+\cdots,
\]
so
\[
\ord_{u=0}\!\big(u^{n-1}P_{F'}(u-1)\big)=r_1
\quad\text{and}\quad
[u^{r_1}]\!\big(u^{n-1}P_{F'}(u-1)\big)=c(F').
\]

For the second summand, write $(u-1)=-1+u$. Since $M(F'')\ge 0$, we have
\[
(u-1)P_{F''}(u-1)=-c(F'')u^{M(F'')}+\cdots,
\]
so
\[
\ord_{u=0}\!\big((u-1)P_{F''}(u-1)\big)=r_2
\quad\text{and}\quad
[u^{r_2}]\!\big((u-1)P_{F''}(u-1)\big)=-c(F'').
\]

Comparing the lowest nonzero powers of $u$ in the sum
\[
P_F(u-1)=\big(u^{n-1}P_{F'}(u-1)\big)+\big((u-1)P_{F''}(u-1)\big)
\]
and using \Cref{lem:ord-sum} gives (a)--(c).
\end{proof}

\begin{remark}\label{rem:MK-min}
Apart from the equal-order case where cancellation may occur, \Cref{thm:JK-recursion-M} yields
\[
M(F)=\min\{ (n-1)+M(F'),\ M(F'')\}.
\]
\end{remark}

\begin{cor}\label{cor:JK-recursion-alpha}
Let $F$ be as in \Cref{not:forest_recursive_deg}. Then
\[
\alpha(F)=\max\{ (n-1)+\alpha(F'),\ 1+\alpha(F'')\}.
\]
\end{cor}

\begin{proof}
If $v$ is not chosen, then we may choose all $n-1$ leaf neighbors
$v_1,\dots,v_{n-1}$ and extend by a maximum independent set of $F'$,
producing an independent set of size $(n-1)+\alpha(F')$.

If $v$ is chosen, then none of $v_1,\dots,v_n$ may be chosen, and the remaining
choices form an independent set in $F''$. Extending by a maximum independent set
in $F''$ produces an independent set of size $1+\alpha(F'')$.

Taking the maximum of these two quantities yields the formula.
\end{proof}

\begin{cor}\label{cor:JK-recursion-deg-h}
Let $F$ be a forest as in \Cref{not:forest_recursive_deg}. Then
\[
\deg h_{R_F/I(F)}(t)=\alpha(F)-M(F),
\]
where $\alpha(F)$ and $M(F)$ may be computed recursively from
\Cref{cor:JK-recursion-alpha} and \Cref{thm:JK-recursion-M}.
\end{cor}

\section{Whiskering and degree formulas}
\label{sec:whiskering}

In this section we record formulas for the degree of the $h$-polynomial of edge and cover ideals under whiskering operations. We treat two versions: uniform $q$-whiskering, where $q$ pendant vertices are attached to every vertex of $G$, and whiskering at a single vertex. The uniform construction yields a closed expression in terms of $\alpha(G)$, whereas the one-vertex version is described in terms of the corresponding invariants of $G$ and $G-v$.

\begin{thm}\label{lem:q-whisker-indep-poly}
Let $G$ be a graph on $n$ vertices, and fix an integer $q\ge 1$.
Let $G_q$ be the graph obtained from $G$ by attaching $q$ whiskers to each
vertex of $G$ and $R_q=\Bbbk[v : v\in G_q]$. Then
\[
\deg h_{R_q/I(G_q)} (t)= q\alpha(G)
\qquad \text{ and } \qquad
\deg h_{R_q/J(G_q)} (t)= n-2 + q\alpha(G).
\]
\end{thm}

\begin{proof}
Recall from \Cref{thm:deg_edge_cover} and  \Cref{lem:deg-via-ord} that, for any graph $H$, one has 
\[
    \deg h_{R/I(H)}(t)=\alpha(H)-M(H) \qquad\text{ and }\qquad
    \deg h_{R/J(H)}(t)=|V(H)|-2-M(H).
\]
So, it suffices find  $M(G_q)$. For this purpose, we focus on the independence polynomial of $G_q$. 

Observe that $G_q$ is the corona product of $G$ and $qK_1$, i.e., $G_q= G \circ qK_1$.  Set $H=qK_1$, so $P_H(x)=(1+x)^q$. Then by the following corona identity from \cite[Theorem~1.2]{levit2003roots} 
\[
P_{G\circ H}(x)=\bigl(P_H(x)\bigr)^{n}P_G\!\left(\frac{x}{P_H(x)}\right),
\]
we have
\[
P_{G_q}(x)=(1+x)^{qn}P_G\!\left(\frac{x}{(1+x)^q}\right) =\sum_{k=0}^{\alpha(G)} g_kx^k(1+x)^{q(n-k)}.
\]
Let $\alpha=\alpha(G)$ and factor the independence polynomial as follows:
\[
P_{G_q}(x)=(1+x)^{q(n-\alpha)}
\Bigg(\sum_{k=0}^{\alpha} g_kx^k(1+x)^{q(\alpha-k)} \Bigg).
\]
Evaluating the remaining factor at $x=-1$, notice that every term with $k<\alpha$
vanishes, leaving only $g_{\alpha}(-1)^{\alpha}\neq 0$.
Thus  $M(G_q)=q(n-\alpha(G))$.

Since $\alpha(G_q)=qn$, we have  
\begin{align*}
    \deg h_{R_q/I(G_q)}(t)&=qn-q(n-\alpha(G))=q\alpha(G),\\
\deg h_{R_q/J(G_q)}(t)&=(q+1)n-2-q(n-\alpha(G))=n-2+q\alpha(G),
\end{align*}
as claimed.
\end{proof}

Whiskering one vertex still yields a useful expression for $\deg h$, but it depends on $G$ and $G\setminus v$ with $v$ the whiskered vertex.

\begin{lem}\label{lem:one-vertex-whisker-hdeg}
Let $G$ be a graph on $n$ vertices and fix a vertex $v\in V(G)$.
Let $G^{v}$ be the graph obtained from $G$ by whiskering $v$.  Set $a=M(G)$ and $b=M(G-v)$.
Then
\[
M(G^{v})=\min\{a,b\}\quad\text{if }a\neq b,
\]
and if $a=b$ then $M(G^{v})\ge a$, with equality unless a cancellation occurs.

Moreover, $\alpha(G^{v})=\max\{\alpha(G),1+\alpha(G-v)\}$.
\end{lem}

\begin{proof}
Let $v'$ be the new vertex where $\{v,v'\} \in E(G^{v})$. 

If we consider a maximum independent set of $G^v$, we can take $v'$ plus any independent set of $G-v$,
giving maximum size $1+\alpha(G-v)$. If we do not include $v'$, we can take a maximum
independent set of $G$, giving size $\alpha(G)$. Hence, we have $\alpha(G^{v})=\max\{\alpha(G),1+\alpha(G-v)\}$.

Any independent set in $G^{v}$ either does not use the new leaf $v'$
(contributing $P_G(x)$), or it uses $v'$, in which case
the remaining vertices form an independent set in $G-v$, contributing
$xP_{G-v}(x)$. Summing the two results yields 
\[
P_{G^{v}}(x)=P_G(x)+xP_{G-v}(x).
\]

Since $a=M(G)$ and $b=M(G-v)$, one has
\[
P_G(x)=(1+x)^a A(x),\qquad P_{G-v}(x)=(1+x)^b B(x),
\]
where $A(-1)\neq 0$ and $B(-1)\neq 0$.
Then
\[
P_{G^{v}}(x)=(1+x)^aA(x)+x(1+x)^bB(x).
\]
If $a<b$, factoring out $(1+x)^a$:
\[
P_{G^{v}}(x)=(1+x)^a\bigl(A(x)+x(1+x)^{b-a}B(x)\bigr),
\]
and evaluating the second factor at $x=-1$ kills the second term, leaving $A(-1)\neq 0$. Thus $M(G^{v})=a$.
The case $b<a$ is symmetric and gives $M(G^{v})=b$. So, if $a\neq b$,
$M(G^{v})=\min\{a,b\}$.

If $a=b$, then
\[
P_{G^{v}}(x)=(1+x)^a\bigl(A(x)+xB(x)\bigr).
\]
So, we have $M(G^{v})\ge a$. Notice that $M(G^{v}) =a$  unless $A(-1)=B(-1)$, in which case the multiplicity is larger.
\end{proof}

\begin{cor}\label{thm:one_whisker}
    Let $G$ be a graph on $n$ vertices and fix a vertex $v\in V(G)$.
Let $G^{v}$ be the graph obtained from $G$ by whiskering $v$.  Then $\deg h_{R^{v}/I(G^{v})}(t)$ and $\deg h_{R^{v}/J(G^{v})}(t)$ are given in terms of $G$ and $G\setminus v$.
\end{cor}
\begin{proof}
    The degree formulas are obtained by substituting $\alpha(G^{v}), |V(G^{v})|=n+1$ and $M(G^v)$  from  Lemma~ \ref{lem:one-vertex-whisker-hdeg}
into $\deg h_{R/I(G^{v})}(t)=\alpha(G^{v})-M(G^{v})$ and
$\deg h_{R/J(G^{v})}(t)=|V(G^{v})|-2-M(G^{v})$.
\end{proof}

\section{Split graphs and diagonal $(\reg,\deg h)$ behavior}
\label{sec:split-graphs}

Split graphs form a well-behaved chordal class for which the independence polynomial admits an explicit description. In this section we exploit a partition of split graphs to compute $P_G(x)$. Hence, we  obtain closed formulas for the invariants $\alpha(G)$ and $M(G)$ that control the degrees of the $h$-polynomials of $R/I(G)$ and $R/J(G)$. As a consequence, cover ideals exhibit particularly rigid behavior in this class: for connected split graphs the pair $\bigl(\reg(R/J(G)),\deg h_{R/J(G)}(t)\bigr)$ always lies on the diagonal, and we determine exactly which diagonal values occur for fixed $n=|V(G)|$.

\begin{defn}
A graph $G$ is a \emph{split graph} if its vertex set admits a partition
\[
V(G)=C\sqcup I
\]
such that $C$ induces a clique and $I$ induces an independent set.
Fix such a partition for the statements below.
Throughout, we assume $G$ is connected and has at least one edge.
\end{defn}

\begin{remark}
It was shown in \cite{split_graphs} that split graphs are exactly $(2K_2,C_4,C_5)$-free. Notice that $2K_2$-free is equivalent to gap-free. 
\end{remark}

We adopt the following notation throughout this section.

\begin{notation}\label{not:split-params}
Let $G$ be a split graph. Fix a split partition  $V(G)=C\sqcup I$  and set $m:=|I|$.
For each $c\in C$, define
\begin{align*}
a(c) &:=\bigl|I\setminus N_G(c)\bigr|, \\
\deg_I(c) &:=\bigl|N_G(c)\cap I\bigr|.
\end{align*}
Thus $m=a(c)+\deg_I(c)$ for each $c\in C$.
Also set
\begin{align*}
    \Delta_I &:=\max \{\deg_I(c): c\in C\},\\
    \delta_I &:=\min \{ \deg_I(c) : c\in C\},\\
    a_{\min} &:=\min \{a(c): c\in C\},\\
    a_{\max}  &:=\max \{a(c): c\in C\}.
\end{align*}
Notice that  $a_{\min}=m-\Delta_I $ and $ a_{\max} =m-\delta_I$.
\end{notation}

\begin{prop}\label{prop:split-independence-polynomial}
Let $G$ be a  split graph  as in Notation~\ref{not:split-params}.
Then its independence polynomial is
\[
P_G(x)=(1+x)^m + x\sum_{c\in C}(1+x)^{a(c)}.
\]
\end{prop}

\begin{proof}
Every independent set is of exactly one of the following types:
\begin{itemize}
\item a subset of $I$, contributing $(1+x)^m$, or
\item it contains a vertex $c\in C$. Since $C$ is a clique, it contains exactly one
vertex of $C$. In addition, it may contain any subset of $I\setminus N_G(c)$.
This contributes $x(1+x)^{a(c)}$ for each $c\in C$.
\end{itemize}
Summing these disjoint contributions gives the formula for the independence polynomial.
\end{proof}

\begin{cor}\label{cor:split-alpha}
With notation as above,
\[
\alpha(G)=\max\Bigl(m,\ 1+a_{\max}\Bigr)
=\max\Bigl(m,\ 1+m-\delta_I\Bigr).
\]
\end{cor}

\begin{proof}
An independent set either lies in $I$ (giving size at most $m$), or it contains a vertex
$c\in C$, in which case it contains exactly one vertex of $C$ and then can include at most
all of $I\setminus N_G(c)$, giving size at most $1+a(c)$. Taking the maximum over $c\in C$
yields the desired formula.
\end{proof}

 As we see below, the invariant $M(G)$ has a nice description for split graphs.

\begin{prop}\label{prop:split-multiplicity}
Let $G$ be a  split graph  as in Notation~\ref{not:split-params}. Then
\[
M(G)=m-\Delta_I.
\]
\end{prop}

\begin{proof}
If $m=0$, then $G$ is a complete graph on $C$ and $P_G(x)=1+|C|x$. This means  $P_G(-1)=1-|C|\neq 0$ and $M(G)=0=m-\Delta_I$.

Assume $m>0$. By Proposition~\ref{prop:split-independence-polynomial}, we know
\[
P_G(x)=(1+x)^m + x\sum_{c\in C}(1+x)^{a(c)}.
\]
Recall that $a_{\min}=m-\Delta_I$. Since $G$ is connected,
some vertex of $C$ has a neighbor in $I$. So, $\Delta_I\ge 1$ which implies that $a_{\min}<m$. Then, we can factor out $(1+x)^{a_{\min}}$ in $P_G(x)$ as follows:
\[
P_G(x)=(1+x)^{a_{\min}}
\Bigl((1+x)^{m-a_{\min}} + x\sum_{c\in C}(1+x)^{a(c)-a_{\min}}\Bigr).
\]
Next, we evaluate the second factor at $x=-1$. Since $m-a_{\min}>0$, the first term vanishes.
In the sum, every term with $a(c)>a_{\min}$ vanishes, while each term with $a(c)=a_{\min}$
contributes $x$. Let $r:=|\{c\in C : a(c)=a_{\min}\}|\ge 1$. Then the second factor in parenthesis evaluates to
\[
(-1)\cdot r\neq 0.
\]
Thus it is not divisible by $(1+x)$.  Therefore $M(G)=a_{\min}=m-\Delta_I$, as claimed.
\end{proof}

\begin{cor}\label{cor:split-deg-h}
Let $G$ be a  split graph  as in Notation~\ref{not:split-params}. Then
\[
\deg h_{R/I(G)}(t)=
\begin{cases}
\Delta_I+1 &: \text{ if } \delta_I=0,\\[4pt]
\Delta_I &: \text{ if } \delta_I>0.
\end{cases}
\]
\end{cor}

\begin{proof}
The result follows from combining Corollary~\ref{cor:split-alpha} and Proposition~\ref{prop:split-multiplicity} with the
identity $\deg h_{R/I(G)}(t)=\alpha(G)-M(G)$.
\end{proof}

\begin{lem}\label{rem:split-iG}
Let $G$ be a  split graph  as in Notation~\ref{not:split-params}. Then
\[
i(G)=m-\Delta_I+1.
\]
\end{lem}

\begin{proof}
 Choose $c_0\in C$ with $\deg_I(c_0)=\Delta_I$. The set
\[
D:=\{c_0\}\ \cup\ \bigl(I\setminus N_G(c_0)\bigr)
\]
is independent and dominating. So $i(G)\le |D|=1+a(c_0)=m-\Delta_I+1$.
Conversely, any independent dominating set contains at most one vertex of $C$. If  an independent dominating contains $c\in C$,
then it must also contain all of $I\setminus N_G(c)$, hence has size at least $1+a(c)\ge 1+a_{\min}=m-\Delta_I+1$.
\end{proof}

\begin{cor}\label{cor:split_cover}
    Let $G$ be a  split graph   as in Notation~\ref{not:split-params}. Then 
    \[
    \deg h_{R/J(G)}(t) = \reg(R/J(G))=\Delta_I+|C|-2.
    \]
\end{cor}

Note that \(\Delta_I+|C|\) is independent of the chosen split partition, since
\(\Delta_I+|C|=n-M(G)\) by Proposition~\ref{prop:split-multiplicity} and \(n=m+|C|\).

\begin{proof}
From \Cref{lem:deg-via-ord} we have $ \deg h_{R/J(G)}(t) = m+|C|-2-M(G)$, and the desired expression follows immediately  from Proposition~\ref{prop:split-multiplicity}. 

For the regularity of the cover ideal, we use \Cref{thm:pdim_chordal} and \Cref{rem:split-iG}.  Thus, 
\[
\reg(R/J(G))= \pdim(R/I(G))-1= m+|C|-m+\Delta_I-2= \Delta_I+|C|-2. \qedhere
\] 
\end{proof}

We are now ready to discuss the realizable $(\reg,\deg h)$ pairs for connected split graphs. 

\begin{thm}\label{cor:split-realizable-pairs}
Fix $n\ge 2$ and let
\[
\mathcal{S}_n
=\Bigl\{\bigl(\reg(R/J(G)),\ \deg h_{R/J(G)}(t)\bigr)\ :\ 
G \text{ is a connected split graph on } n \text{ vertices}\Bigr\}.
\]
Then $\mathcal{S}_n$ consists precisely of diagonal pairs $(q,q)$ with $q_{\min}(n)\le q\le n-2$  where 
\[
q_{\min}(n):=\min_{1\le c\le n-1}\Bigl(c+\Bigl\lceil \frac{n}{c}\Bigr\rceil-3\Bigr).
\]
\end{thm}

\begin{figure}[ht]
\centering
\begin{tikzpicture}[scale=0.65]

    \draw[->] (0,0) -- (8.5,0) node[right] {$\text{reg}$};
    \draw[->] (0,0) -- (0,8.5) node[above] {$\text{deg}$};

    \foreach \x in {1,2,3,4,5,6,7,8} {
        \draw (\x,-0.1) -- (\x,0.1);
        \node[below] at (\x,-0.2) {\x};
    }
    
    \foreach \y in {1,2,3,4,5,6,7,8} {
        \draw (-0.1,\y) -- (0.1,\y);
        \node[left] at (-0.15,\y) {\y};
    }

    \foreach \x/\y in {
        4/3,4/4,4/5,4/6,4/7,
        5/5,5/6,5/7,
        6/6,6/7,
        7/7
    }{
        \filldraw[teal] (\x,\y) circle (3pt);
    }

    \foreach \x/\y in {3/3,4/4,5/5,6/6,7/7}{
        \draw[orange, very thick] (\x-0.14,\y-0.14) -- (\x+0.14,\y+0.14);
        \draw[orange, very thick] (\x-0.14,\y+0.14) -- (\x+0.14,\y-0.14);
    }

    \begin{scope}[shift={(5.1,1.1)}]
        \filldraw[teal] (0,0) circle (3pt);
        \node[anchor=west] at (0.3,0) {radius $\leq 2$ trees};
        \draw[orange, very thick] (0,-0.6) ++(-0.14,-0.14) -- ++(0.28,0.28);
        \draw[orange, very thick] (0,-0.6) ++(-0.14,0.14) -- ++(0.28,-0.28);
        \node[anchor=west] at (0.3,-0.6) {split graphs};
    \end{scope}

    \foreach \x/\y in {5/4,6/4,6/5,7/5,7/6}{
        \draw[blue, very thick, fill=white] (\x,\y) circle (3pt);
    }
\end{tikzpicture}

\caption{All possible $(\reg,\deg h)$ values for connected graphs on $n=9$ vertices.
Filled points are realized by split graphs and radius--$\leq 2$ trees; hollow points are not.}
\label{fig:reg_deg_9_split_tree_9}
\end{figure}
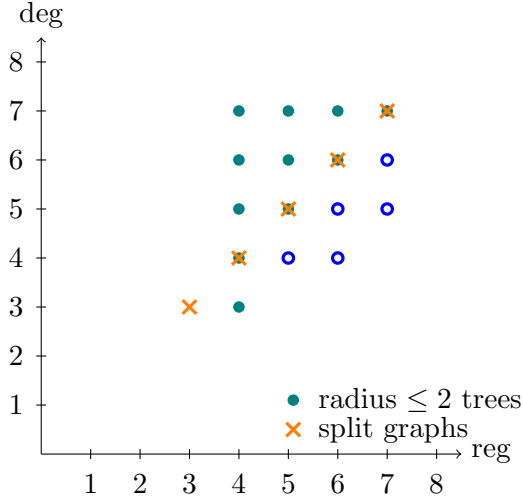

\begin{proof}
Let $G$ be connected split graph with $V(G)=C\sqcup I$, where $|C|=c$ and $|I|=m=n-c$,
and let $\Delta_I=\max\{\deg_I(x):x\in C\}$ as in Notation~\ref{not:split-params}.

If $c=1$, then $G$ is a star with center $C$ and $\Delta_I=m=n-1$. Hence
$\Delta_I+c-2=n-2$.  Thus this case only realizes the top diagonal value and we may assume $c\ge 2$.

By Corollary~\ref{cor:split_cover},
\[
\reg(R/J(G))=\deg h_{R/J(G)}(t)=\Delta_I+c-2,
\]
so every realizable pair is diagonal. We claim that  $ q_{\min} (n) \leq  \Delta_I+c-2 \leq n-2$.  The upper bound is immediate since $\Delta_I\le m$ implies $\Delta_I+c-2\le m+c-2=n-2$.

Since $G$ is connected, every vertex of $I$ has a neighbor in $C$. Hence the number
of edges between $C$ and $I$, denoted by $e(C,I)$, satisfies $e(C,I)\ge m$. On the other hand,
$e(C,I)=\sum_{x\in C}\deg_I(x)\le c\Delta_I$. So, we have $c\Delta_I\ge m$. Thus
\[
\Delta_I\ \ge\ \Bigl\lceil \frac{m}{c}\Bigr\rceil
=\Bigl\lceil \frac{n-c}{c}\Bigr\rceil
=\Bigl\lceil \frac{n}{c}\Bigr\rceil-1.
\]
Therefore
\[
\Delta_I+c-2\ \ge\ c+\Bigl\lceil \frac{n}{c}\Bigr\rceil-3,
\]
and taking the minimum over $2\le c\le n-1$ gives $\Delta_I+c-2\ge q_{\min}(n)$.

Conversely, fix $c\in \{2,\ldots, n-1\}$ and an integer $q$ with
$c+\lceil n/c\rceil-3\le q\le n-2$. Set $m=n-c$ and $\Delta_I=q-c+2$.
Then $\lceil m/c\rceil\le \Delta_I\le m$. Choose nonnegative integers
$s_1,\dots,s_c$ with $\sum_{j=1}^c s_j=m$, $s_1=\Delta_I$, and $s_j\le \Delta_I$ for all $j$.
Let $C=\{x_1,\dots,x_c\}$ induce a clique, let $I$ be independent, partition
$I=I_1\sqcup\cdots\sqcup I_c$ with $|I_j|=s_j$, and join $x_j$ to every vertex of $I_j$.
Then $G$ is a connected  split graph and it satisfies $\max_{x\in C}\deg_I(x)=\Delta_I$. Hence
\[
\reg(R/J(G))=\deg h_{R/J(G)}(t)=\Delta_I+c-2=q.
\]
Thus every $q$ in the stated range is observed by a split graph, proving the claim.
\end{proof}

\section{Graphs with $\reg(R/J(G))=n-2$ and varying $\deg h$}
\label{sec:reg=n-2}

In this section we study graphs $G$ on $n$ vertices whose cover ideal attains the maximal possible regularity,
\[
\reg(R/J(G))=n-2.
\]
and we construct families for which the degree $\deg h_{R/J(G)}(t)$ varies. Our main tool is the cone operation:
if $G=H*K_1$ is obtained by coning over a graph $H$ on $n-1$ vertices, then $\pdim(R/I(G))=n-1$, hence
$\reg(R/J(G))=n-2$. Since the degree is governed by $M(G)$,
\[
\deg h_{R/J(G)}(t)=n-2-M(G),
\]
to vary $\deg h$ while keeping $\reg(R/J(G))$ fixed, it suffices to build graphs with controlled $M(G)$.
We do this by prescribing independence polynomials for suitable chordal graphs $H$ (built from a split graph and a star $K_{1,r}$), and then coning over $H$.

We use the following remark repeatedly throughout this section.

\begin{remark}\label{rem:setting}
Let $H$ be a graph on $(n-1)$-vertices and $G=H\ast K_1$ be the graph obtained from $H$ by coning with a new vertex $u$. 
In addition, since the cone vertex $u$ is adjacent to every vertex of $H$, the only independent set containing $u$
is $\{u\}$. So
\[
P_G(x)=P_H(x)+x.
\]
If $H = K \sqcup L$, then $P_H(x)= P_K(x) P_L(x)$.

Lastly,  the independence polynomial of a star graph $K_{1,r}$ where $r\geq 1$ is
$$P_{K_{1,r}}(x)=(1+x)^r+x.$$
\end{remark}

\begin{construction}\label{cons:Bk}
Fix an integer $k\ge 1$.
Let
\[
C=\{c_1,c_2,\dots,c_{k+1}\}
\qquad\text{and}\qquad
I=\{z_1,\dots,z_{k-1}\},
\]
where $I=\emptyset$ if $k=1$.
Define a graph $B_k$ on the vertex set $V(B_k)=C\sqcup I$ such that
\begin{itemize}
\item $C$ spans a clique $K_{k+1}$;
\item $I$ is an independent set;
\item for each $1\le j\le k-1$, the neighborhood of $z_j$ is
\[
N_{B_k}(z_j)=\{c_1,c_2,\dots,c_{j}\}.
\]
\end{itemize}
Equivalently, the neighborhoods are nested
\[
N(z_1)\subset N(z_2)\subset \cdots \subset N(z_{k-1}).
\]
In particular, $B_k$ is a  split graph  (hence chordal) and $|V(B_k)|=2k$.
\end{construction}

\begin{figure}[ht]
\centering
\scalebox{0.8}{
\begin{tikzpicture}[thick, line cap=round, line join=round]
\tikzset{
  v/.style={circle, fill=black, inner sep=1.8pt},
  lab/.style={draw=none, fill=none, inner sep=0pt}
}

\begin{scope}[xshift=0cm]
  \node[v] (bOnec1) at (0,0) {};
  \node[v] (bOnec2) at (0,2) {};
  \draw (bOnec1)--(bOnec2);

  \node[lab, below=2mm] at (bOnec1.south) {$c_1$};
  \node[lab, above=2mm] at (bOnec2.north) {$c_2$};
  \node[lab, below=12mm] at (0,0) {$B_1$};
\end{scope}

\begin{scope}[xshift=4.2cm]
  \node[v] (bTwoc1) at (0,0) {};
  \node[v] (bTwoc2) at (0,2) {};
  \node[v] (bTwoc3) at (-1.8,1) {};
  \node[v] (bTwoz1) at (2.2,0) {};

  \draw (bTwoc1)--(bTwoc2)--(bTwoc3)--(bTwoc1);
  \draw (bTwoc1)--(bTwoz1);

  \node[lab, below=2mm] at (bTwoc1.south) {$c_1$};
  \node[lab, above=2mm] at (bTwoc2.north) {$c_2$};
  \node[lab, left=2mm]  at (bTwoc3.west)  {$c_3$};
  \node[lab, below=2mm] at (bTwoz1.south) {$z_1$};
  \node[lab, below=12mm] at (0,0) {$B_2$};
\end{scope}

\begin{scope}[xshift=8.6cm]
  \node[v] (bThreec4) at (0,0) {};
  \node[v] (bThreec3) at (0,2) {};
  \node[v] (bThreec1) at (2.6,0) {};
  \node[v] (bThreec2) at (2.6,2) {};
  \node[v] (bThreez1) at (5.8,0) {};
  \node[v] (bThreez2) at (5.8,2) {};

  \draw (bThreec4)--(bThreec3)--(bThreec2)--(bThreec1)--(bThreec4);
  \draw (bThreec3)--(bThreec1);
  \draw (bThreec2)--(bThreec4);

  \draw (bThreec1)--(bThreez1);
  \draw (bThreec1)--(bThreez2);
  \draw (bThreec2)--(bThreez2);

  \node[lab, below=2mm] at (bThreec1.south) {$c_1$};
  \node[lab, above=2mm] at (bThreec2.north) {$c_2$};
  \node[lab, above=2mm] at (bThreec3.north) {$c_3$};
  \node[lab, below=2mm] at (bThreec4.south) {$c_4$};
  \node[lab, below=2mm] at (bThreez1.south) {$z_1$};
  \node[lab, above=2mm] at (bThreez2.north) {$z_2$};

  \node[lab, below=12mm] at (2.6,0) {$B_3$};
\end{scope}

\end{tikzpicture}
}
\caption{$B_k$ for $k=1,2,3$ from Construction~\ref{cons:Bk} from left to right}
\label{fig:B_k}
\end{figure}
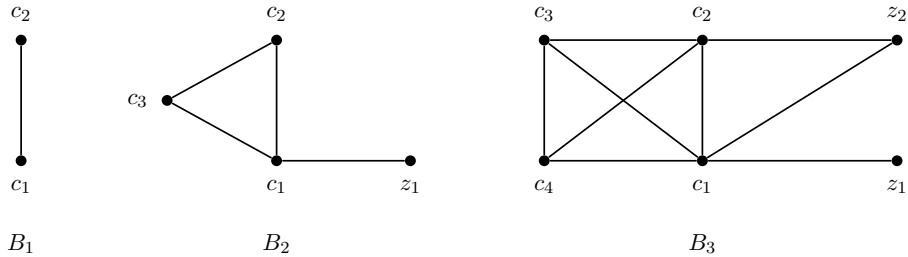

\begin{lem}
Let $B_k$ be as in Construction~\ref{cons:Bk}.
Then 
\[
P_{B_k}(x)=2(1+x)^k-1.
\]
\end{lem}

\begin{proof}
Independent sets in $B_k$ fall into two types.

\begin{enumerate}
    \item[(i)] Consider independent sets avoiding the clique $C$.
Since $I$ is independent, any subset of $I$ is independent, contributing $(1+x)^{k-1}$.

    \item[(ii)] Consider independent sets containing a vertex of the clique $C$. If an independent set contains $c_i\in C$, it cannot contain any other vertex of $C$. It may contain those $z_j$ that are not adjacent to $c_i$. By construction, $z_j$ is adjacent to $c_i$ if and only if $i\le j$. Thus the allowable vertices among $I$ are $z_1,\dots,z_{\min(i,k-1)}$, giving $(1+x)^{\min(i-1,k-1)}$ choices for the subset of $I$. Therefore the total contribution from choosing $c_i$ is $x(1+x)^{\min(i-1,k-1)}$. 
\end{enumerate}

Summing over $i=1,\dots,k+1$ yields
\[
P_{B_k}(x)
= (1+x)^{k-1} + x\!\left(\sum_{j=0}^{k-2} (1+x)^j + 2(1+x)^{k-1}\right).
\]
Since
\[
\sum_{j=0}^{k-2}(1+x)^j=\frac{(1+x)^{k-1}-1}{(1+x)-1},
\]
we obtain
\[
P_{B_k}(x)
= (1+x)^{k-1} + \bigl((1+x)^{k-1}-1\bigr) + 2x(1+x)^{k-1}
= 2(1+x)^k-1,
\]
as claimed.
\end{proof}

\begin{construction}\label{cons:Gkr}
Fix integers $k\ge 1$ and $r\ge k$.
Let $H:=B_k\ \sqcup\ K_{1,r}$ and $G_{k,r}$ be the cone over $H$, i.e.
\[
G_{k,r}\ :=\  H  * K_1,
\]
obtained by adding a new vertex $u$ adjacent to every vertex of $H$.
Then $G_{k,r}$ is a connected  chordal graph with $|V(G_{k,r})|=2k+r+2$.
\end{construction}

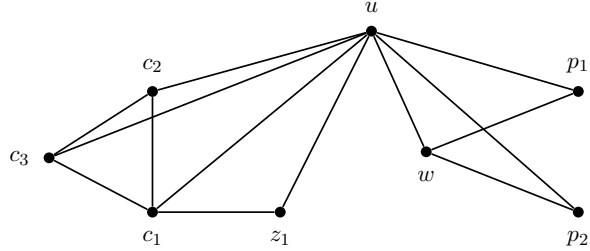
\begin{figure}[ht]
\centering
\scalebox{0.8}{%
\begin{tikzpicture}[thick, line cap=round, line join=round]
\tikzset{
  v/.style={circle, fill=black, inner sep=1.8pt},
  lab/.style={draw=none, fill=none, inner sep=0pt}
}

\node[v] (c1) at (0,0) {};
\node[v] (c2) at (0,2) {};
\node[v] (c3) at (-1.7,0.9) {};
\node[v] (z1) at (2.1,0) {};

\draw (c1)--(c2)--(c3)--(c1);
\draw (c1)--(z1);

\node[lab, below=2mm] at (c1.south) {$c_1$};
\node[lab, above=2mm] at (c2.north) {$c_2$};
\node[lab, left=2mm]  at (c3.west)  {$c_3$};
\node[lab, below=2mm] at (z1.south) {$z_1$};

\node[v] (w)  at (4.5,1) {};
\node[v] (p1) at (7.0,2) {};
\node[v] (p2) at (7.0,0) {};

\draw (w)--(p1);
\draw (w)--(p2);

\node[lab, below=2mm] at (w.south)  {$w$};
\node[lab, above=2mm] at (p1.north) {$p_1$};
\node[lab, below=2mm] at (p2.south) {$p_2$};

\node[v] (u) at (3.6,3.0) {};
\node[lab, above=2mm] at (u.north) {$u$};

\draw (u)--(c1);
\draw (u)--(c2);
\draw (u)--(c3);
\draw (u)--(z1);
\draw (u)--(w);
\draw (u)--(p1);
\draw (u)--(p2);

\end{tikzpicture}%
}
\caption{The graph \(G_{2,2}\) from Construction~\ref{cons:Gkr}: the cone over \(B_2\sqcup K_{1,2}\)}
\label{fig:G_22}
\end{figure}

Next, we compute the multiplicity of this family.

\begin{lem}\label{lem:Gkr-M=k}
Let $G_{k,r}$ be as in Construction~\ref{cons:Gkr}.
Then 
\[
M(G_{k,r})=k.
\]
\end{lem}

\begin{proof}
It follows from Remark~\ref{rem:setting} that
\begin{align*}
    P_{G_{k,r}}(x)
    &=P_{B_k}(x)P_{K_{1,r}}(x) +x \\
    &=\Big( 2(1+x)^k-1\Big) \Big(  (1+x)^r+x\Big) +x \\
    &= 2(1+x)^{k+r} +2x(1+x)^k -(1+x)^r\\
    &= (1+x)^k \underbrace{\Big(  2 (1+x)^r+2x -(1+x)^{r-k}\Big)}_{Q(x)}
\end{align*}
since $r\geq k$. Hence  $M(G_{k,r})=k$ since $Q(-1)=-2\neq 0$.
\end{proof}

We can provide a lower bound and an upper bound on the degree of the $h$-polynomial for this class of graphs.

\begin{lem}\label{lem:lower-bound-Gkr}
Fix $n$. Let $G_{k,r}$ be as in
Construction~\ref{cons:Gkr} with $|V(G_{k,r})|=n$. Then
\[
n-3\geq \deg h_{R/J(G_{k,r})}(t)\ge\
\left\lceil\frac{2(n-2)}{3}\right\rceil.
\]
\end{lem}
\begin{proof}
    By \Cref{lem:Gkr-M=k} we have $M(G_{k,r})=k\geq 1$. Thus \[
    \deg h_{R/J(G_{k,r})}(t) = n-2-M(G_{k,r}) = n-2-k \leq n-3,
    \]
    as desired. It remains to show the lower bound.

Since $|V(B_k)|=2k$ and $|V(K_{1,r})|=r+1$ with the constraint $r\ge k$, we have
\[
n  =  |V(G_{k,r})|
 =  2k+(r+1)+1
 \ge  2k+(k+1)+1
 =  3k+2.
\]
Consequently,
\[
1\leq k = M(G)  \le  \left\lfloor\frac{n-2}{3}\right\rfloor.
\] 
Then 
\[
\deg h_{R/J(G_{k,r})}(t)
= n-2-M(G_{k,r})
\ \ge\
n-2-\left\lfloor\frac{n-2}{3}\right\rfloor = \left\lceil\frac{2(n-2)}{3}\right\rceil.\qedhere
\]
\end{proof}

\begin{ex}
    Let $n=8$. The bound in the proof of \Cref{lem:lower-bound-Gkr} gives $k\le \lfloor(8-2)/3\rfloor=2$, and both values occur.
Indeed, solving $n=2k+r+2$ gives $r=8-2k-2$:
\[
(k,r)=(1,4)\quad\text{and}\quad (k,r)=(2,2),
\]
both satisfying $r\ge k$. So, the corresponding regularity and degree values are
\[
(\reg,\deg h)=(6,5)\ \ \text{(from $k=1$)}\qquad\text{and}\qquad
(\reg,\deg h)=(6,4)\ \ \text{(from $k=2$)}.
\]
Our Macaulay2 computations for all possible $(\reg,\deg h)$ pairs for $n=8$ show that  these  are all the $(\reg, \deg h )$ pairs  when $\reg = 6=n-2$ (apart from $(6,6)$  which is realized by a star).

Let $n=9$. One has $k\le \lfloor(9-2)/3\rfloor=2$. So, this family reaches down to
$(\reg,\deg h)=(7,5)$ but not below. Again, our Macaulay2 computations for all possible $(\reg,\deg h)$ pairs for $n=9$ shows that this is the smallest pair we can have on the $(\reg, \deg h)$  when $\reg = 7=n-2$ (see \Cref{fig:reg_deg_9_split_tree_9}).
\end{ex}

\section{A chordal family on the line 
$\deg h=\reg-1$}\label{sec:deg=reg-1}

In this section we construct, for fixed $n$, a family of connected chordal graphs whose cover ideals realize the line one step below the diagonal in the $(\reg,\deg h)$-plane, namely
\[
\deg h_{R/J(G)}(t)=\reg(R/J(G))-1.
\]
It follows from \Cref{thm:pdim_chordal}  that, for a chordal graph $G$, we have
\[
\reg(R/J(G))=n-i(G)-1. 
\]
We also have the following for any graph
\[
\deg h_{R/J(G)}(t)=n-2-M(G).
\]
So the condition $\deg h =\reg-1$ is equivalent to requiring $M(G)=i(G)$.
Our construction produces a radius--$2$ chordal graph $H_{n,p}$ (obtained from a radius--$2$ tree by gluing triangles along a fixed edge) for which
\[
i(H_{n,p})=M(H_{n,p})=p.
\]
Varying $p$ then yields, for fixed $n$, the full range of diagonal-adjacent pairs
\[
\bigl(\reg(R/J(H_{n,p})),\ \deg h_{R/J(H_{n,p})}(t)\bigr)=(r,r-1)
\]
for $\left\lfloor\frac{n}{2}\right\rfloor\le r\le n-3$.

\begin{construction}\label{cons:belowdiag-chordal}
Fix integers $p\ge 3$ and $n\ge 2p+1$, and set $k:=n-(2p+1)\ge 0$.
We define a connected chordal graph $H_{n,p}$ on $n$ vertices as follows.

The vertex set of $H_{n,p}$ is
\[
V(H_{n,p})=\{c,v_1,v_2\}\ \sqcup\ \{u_1,\dots,u_{p-1}\}\ \sqcup\ \{y_1,\dots,y_{p-2}\}\ \sqcup\ \{z\}\ \sqcup\ \{x_1,\dots,x_k\}.
\]
The edges of $H_{n,p}$ are
\begin{itemize}
    \item $cv_1,\ cv_2, cu_j $ for each $1\le j\le p-1$,
    \item $v_1y_\ell $ for each $1\le \ell\le p-2$ and $v_2z$
    \item    $cx_i$ and  $v_1x_i$ for each $1\le i\le k$ .
\end{itemize}

Equivalently, $H_{n,p}$ is obtained from the radius--$2$ tree with center $c$
(by attaching $p-1$ leaves at $c$, $p-2$ leaves at $v_1$, and one leaf at $v_2$)
and then gluing $k$ triangles along the edge $cv_1$ (each triangle has vertices $\{c,v_1,x_i\}$).
\end{construction}

\begin{remark}
For all $p\ge 3$ and $n\ge 2p+1$, the graph $H_{n,p}$ from Construction~\ref{cons:belowdiag-chordal}
is connected by construction. In addition, $H_{n,p}$ is chordal.
\end{remark}

\begin{lem}\label{lem:belowdiag-chordal-i}
Let  $H_{n,p}$ be the graph from Construction~\ref{cons:belowdiag-chordal}. Then the independent domination number of $H_{n,p}$ is
\[
i(H_{n,p})=p.
\]
\end{lem}

\begin{proof}
Set
\[
D:=\{c\}\ \cup\ \{y_1,\dots,y_{p-2}\}\ \cup\ \{z\}.
\]
This set is independent since each $y_\ell$ is adjacent only to $v_1$, and $z$ is adjacent only to $v_2$,
while $c$ is adjacent only to $v_1,v_2,u_j,x_i$. It is dominating because
$c$ dominates $v_1,v_2,u_1,\dots,u_{p-1},x_1,\dots,x_k$, each $y_\ell$ dominates $v_1$, and $z$ dominates $v_2$.
Hence $i(H_{n,p})\le |D|=p$.

For the reverse inequality, let $D'$ be any independent dominating set.
Consider two cases.

\begin{enumerate}
    \item[Case 1:] Suppose $c\in D'$.
Then $v_1,v_2\notin D'$. Each $y_\ell$ must be dominated, but $N(y_\ell)=\{v_1\}$.
So, $y_\ell\in D'$ for all $\ell$. Similarly $N(z)=\{v_2\}$ forces $z\in D'$.
Thus $|D'|\ge 1+(p-2)+1=p$.
    \item[Case 2:] Suppose $c\notin D'$.
Each $u_j$ has $N(u_j)=\{c\}$. So $u_j\in D'$ for all $j$, giving $|D'|\ge p-1$.
Also, the leaf $z$ forces $z\in D'$ or $v_2\in D'$.
Finally, since $v_1\notin N[u_j]$ for any $j$, the vertices $y_1,\dots,y_{p-2}$ must be dominated
either by including $v_1$ or by including all $y_\ell$.
In either subcase we add at least one more vertex beyond the $p-1$ leaves $u_j$,
and we also need $z$ or $v_2$. Hence $|D'|\ge (p-1)+2=p+1$ in this case.
\end{enumerate}

Thus, every independent dominating set has size at least $p$. Therefore, $i(H_{n,p})=p$.
\end{proof}

\begin{lem}\label{lem:belowdiag-chordal-M}
Let  $H_{n,p}$ be the graph from Construction~\ref{cons:belowdiag-chordal}. Then, we have
\[
M(H_{n,p})=p.
\]
\end{lem}

\begin{proof}
Write $U=\{u_1,\dots,u_{p-1}\}$, $Y=\{y_1,\dots,y_{p-2}\}$, and $X=\{x_1,\dots,x_k\}$.
We count independent sets by their intersection with $\{c,v_1,v_2\}$.

\begin{enumerate}
    \item[(i] \emph{Independent sets containing $c$.}
If $c\in S$, then $S$ cannot contain $v_1,v_2$, any $u_j$, or any $x_i$.
The vertices in $Y\cup\{z\}$ are then free. This contributes the following towards to independence polynomial:
\[
x(1+x)^{|Y|}(1+x)=x(1+x)^{p-1}.
\]
    \item[(ii)] \emph{Independent sets containing $v_1$ but not $c$.}
If $v_1\in S$ and $c\notin S$, then $Y\cup X$ is forbidden, while $U$ is free.
Also $v_2$ may be chosen or not (it is not adjacent to $v_1$).
If $v_2\notin S$, then $z$ is free; if $v_2\in S$, then $z$ is forbidden.
So the contribution is
\[
x(1+x)^{p-1}(1+x)\ +\ x^2(1+x)^{p-1}
= x(1+x)^p + x^2(1+x)^{p-1}.
\]

    \item[(iii)] \emph{Independent sets containing $v_2$ but neither $c$ nor $v_1$.}
If $v_2\in S$ and $c,v_1\notin S$, then $z$ is forbidden, while $U$, $Y$, and $X$ are all free.
Thus this contributes
\[
x(1+x)^{|U|+|Y|+|X|}
= x(1+x)^{(p-1)+(p-2)+k}
= x(1+x)^{2p+k-3}.
\]

    \item[(iv)] \emph{Independent sets containing none of $\{c,v_1,v_2\}$.}
Then $U,Y,X,$ and $z$ are all free, contributing
\[
(1+x)^{|U|+|Y|+|X|+1}=(1+x)^{2p+k-2}.
\]
\end{enumerate}

Summing (i)--(iv) gives
\[
\begin{aligned}
P_{H_{n,p}}(x)
&=(1+x)^p\underbrace{\Bigl(2x  +  x(1+x)^{p+k-3}  +  (1+x)^{p+k-2}\Bigr)}_{Q(x)}.
\end{aligned}
\]
Observe that $Q(-1)=-2$ if $p+k-3\ge 1$ (i.e.\ $p\ge 4$ or $k\ge 1$),
and $Q(-1)=-3$ in the remaining case $(p,k)=(3,0)$. In all cases $Q (-1)\neq 0$.
Hence  $M(H_{n,p})=p$.
\end{proof}

\begin{remark}
The case $p=2$ can also be handled by the same construction, but one needs $k\ge 1$ (equivalently $n\ge 6$).
In that case the same case-counting argument shows that  $i(H_{n,2})=M(H_{n,2})=2$ and $(\reg,\deg h)=(n-3,n-4)$.
\end{remark}

So, $H_{n,p}$ realizes the line $d=r-1$ (one step below the diagonal).

\begin{cor}
Fix an integer $n\ge 6$. For each integer $2\ \le\ p\ \le\ \left\lfloor\frac{n-1}{2}\right\rfloor$, 
let $H_{n,p}$ be the graph from Construction~\ref{cons:belowdiag-chordal}.
Then 
\[
\bigl(\reg(R/J(H_{n,p})),\ \deg h_{R/J(H_{n,p})}(t)\bigr)
=\bigl(n-p-1,\ n-p-2\bigr).
\]
In particular, for fixed $n$ these points lie on the line $\deg h=\reg-1$ and form the  set
\[
\Bigl\{(r,r-1):\ \left\lfloor\frac{n}{2}\right\rfloor \le r \le n-3\Bigr\}.
\]
\end{cor}

\begin{proof}
By Lemmas~\ref{lem:belowdiag-chordal-i} and \ref{lem:belowdiag-chordal-M},
we have $i(H_{n,p})=M(H_{n,p})=p$ for all $p$ in the stated range.
For chordal graphs,
\[
\reg(R/J(G))=n-i(G)-1
\qquad\text{and}\qquad
\deg h_{R/J(G)}(t)=n-2-M(G).
\]
Substituting $i(H_{n,p})=M(H_{n,p})=p$ gives
\[
\reg(R/J(H_{n,p}))=n-p-1,
\qquad
\deg h_{R/J(H_{n,p})}(t)=n-p-2,
\]
so $\deg h=\reg-1$.

Finally, letting $r:=n-p-1$, as $p$ ranges over
$2\le p\le \lfloor (n-1)/2\rfloor$ we obtain
\[
\left\lfloor\frac{n}{2}\right\rfloor
= n-\left\lfloor\frac{n-1}{2}\right\rfloor-1
\le r \le n-3.
\]
Hence the realized pairs are exactly
$\{(r,r-1): \lfloor n/2\rfloor \le r \le n-3\}$.
\end{proof}

As the following discussion/example indicates that this construction complements the previous section with a step towards filling all the points below the diagonal.

\begin{ex}
Note that for $n=9$, this construction covers all $(r,r-1)$ where $4\leq r\leq 6$ and this corresponds to all the $(\reg,\deg h)$ pairs we were not able to cover so far on this line below the diagonal (see \Cref{fig:reg_deg_9_split_tree_9}). Notice that we have a construction to cover $(\reg,\deg h)$ pairs when $r=n-2$. 
\end{ex}

\section{Block graphs lie on or above the line $\deg h=\reg-1$}
\label{sec:block_graphs}

In this section we show that block graphs exhibit a strong restriction in the $(\reg,\deg h)$-plane for cover ideals: if $G$ is a connected block graph and we set
\[
r:=\reg(R/J(G))\qquad\text{and}\qquad d:=\deg h_{R/J(G)}(t),
\]
then necessarily $d\ge r-1$. Equivalently, block graphs cannot contribute points strictly below the line $d=r-1$.
Since block graphs are chordal, one has $\reg(R/J(G))=n-i(G)-1$, while
\[
\deg h_{R/J(G)}(t)=n-2-M(G).
\]
So the inequality $d\ge r-1$ is precisely the statement $M(G)\le i(G)$.
We prove this by exploiting the clique-tree structure of block graphs: removing a leaf maximal clique yields a  recursion for the independence polynomial, and comparing the resulting orders at $x=-1$ with a parallel recursion for the independent domination number gives $M(G)\le i(G)$ by induction on the number of maximal cliques.

\begin{defn}
A  graph $G$ is a \emph{block graph} if every biconnected component (block) of $G$
is a clique. Equivalently, $G$ is chordal and any two distinct maximal cliques intersect
in at most one vertex.
\end{defn}

\begin{defn}
Let $G$ be a chordal graph. A \emph{clique tree} of $G$ is a tree $T$ whose vertices are the maximal
cliques of $G$ and such that for every vertex $v\in V(G)$, the set of maximal cliques containing $v$
forms a connected subtree of $T$.

If $G$ is a connected block graph, then $G$ admits a clique tree $T$.
Moreover, if two maximal cliques $K$ and $K'$ are adjacent in $T$, then
$K\cap K'$ consists of a single vertex $s$, and this vertex is a cut vertex of $G$, i.e.\ $G\setminus\{s\}$ is disconnected.
A maximal clique corresponding to a leaf of $T$ is called a \emph{leaf maximal clique}.
\end{defn}

\begin{notation}\label{not:block_graphs}
  Let $G$ be a connected block graph that is not a clique, and fix a clique tree of $G$.
Let $K$ be a leaf maximal clique, and let $K'$ be its unique neighbor in the clique tree. Set
\[
s:=K\cap K',\qquad C:=K\setminus\{s\},\qquad r:=|C|\ge 1,
\]
and define the induced subgraphs
\[
H:=G\setminus K,\qquad L:=G\setminus N_G[s].
\]  
\end{notation}

The following leaf-clique recursion is used in the proof of \Cref{thm:block-M-leq-i}.

\begin{lem}\label{lem:leaf-clique-recursion}
Let $G$ be a connected block graph and use Notation \ref{not:block_graphs}.
Then
\[
P_G(x)\ =\ (1+rx)P_H(x)\ +\ xP_L(x).
\]
\end{lem}

\begin{proof}
We partition the independent sets of $G$ into three disjoint classes according to how they meet $K$.

\begin{enumerate}
    \item \emph{Independent sets avoiding $K$.}
These are exactly the independent sets of the induced subgraph $H=G\setminus K$, contributing $P_H(x)$.

    \item \emph{Independent sets meeting $C$.}
Since $K$ is a clique, an independent set can contain \emph{at most one} vertex of $K$. Hence it contains
at most one vertex of $C$ and cannot contain $s$.
Moreover, because $K$ is a leaf in the clique tree, every vertex $v\in C$ belongs to no maximal clique
other than $K$. In particular, $v$ has no neighbors outside $K$. Therefore, after choosing $v\in C$,
the rest of the independent set can be \emph{any} independent set of $H$.
Thus this class contributes
\[
\sum_{v\in C} xP_H(x)\ =\ r xP_H(x).
\]
    \item \emph{Independent sets containing $s$.}
If an independent set contains $s$, then it avoids $N_G[s]$, and its remaining vertices form an
independent set of $L=G\setminus N_G[s]$. This contributes $xP_L(x)$.
\end{enumerate}

Adding the contributions from (1)--(3) gives
\[
P_G(x)\ =\ P_H(x)\ +\ r xP_H(x)\ +\ xP_L(x)\ =\ (1+rx)P_H(x)+xP_L(x),
\]
as claimed.
\end{proof}

Now we are for the main result of this section. 

\begin{thm}\label{thm:block-M-leq-i}
Let $G$ be a connected block graph. Then
\[
M(G)\ \le\ i(G).
\]
Consequently, 
\[
\reg(R/J(G))-\deg h_{R/J(G)}(t)
 = M(G)-i(G)+1 \le 1.
\]
\end{thm}

\begin{proof}
We induct on the number of maximal cliques of $G$.

If $G$ is a clique, then $P_G(x)=1+|V(G)|x$. So $M(G)=0$ while $i(G)=1$. Hence $M(G)\le i(G)$.

Assume now that $G$ is not a clique. Fix a leaf maximal clique $K$ in the clique tree as in Notation \ref{not:block_graphs} where $K'$ be
its unique neighbor, and 
\[
s:=K\cap K',\qquad C:=K\setminus\{s\},\qquad r:=|C|\ge 1,
\]
together with the induced subgraphs
\[
H:=G\setminus K,\qquad L:=G\setminus N_G[s].
\]

Since $K$ is a clique, any independent dominating set $D$ meets $K$ in at most one vertex.
Moreover $D$ must dominate every vertex of $C$, and the only vertices that can dominate $C$ lie in $K$.
Thus $D$ contains either (a) a vertex $v\in C$, or (b) the cut vertex $s$.

If $v\in C$ is chosen, then $v$ has no neighbors outside $K$. So $D\setminus\{v\}$ is an independent
dominating set of $H$. Hence $|D|\ge 1+i(H)$.
If $s$ is chosen, then no vertex of $N_G[s]$ may lie in $D$, and $D\setminus\{s\}$ is an independent
dominating set of $L$. Hence $|D|\ge 1+i(L)$.
Conversely, choosing any $v\in C$ together with an independent dominating set of $H$ gives an
independent dominating set of $G$. Similarly, choosing $s$ together with an independent dominating set of $L$
also gives one. Therefore
\begin{equation}\label{eq:i-rec}
i(G) = 1+\min\{i(H),i(L)\}.
\end{equation}

Set $u:=x+1$ and $Q_F(u):=P_F(u-1)$, so that $M(F)=\ord_{u=0}Q_F(u)$.
The leaf-clique recursion from Lemma \ref{lem:leaf-clique-recursion} gives
\[
P_G(x)=(1+rx)P_H(x)+xP_L(x),
\]
and substituting $x=u-1$ yields
\begin{equation}\label{eq:Q-rec}
Q_G(u)=\bigl(ru-(r-1)\bigr)Q_H(u)+(u-1)Q_L(u).
\end{equation}
Let $m:=\min\{M(H),M(L)\}$. Write $Q_H(u)=u^m h(u)$ and $Q_L(u)=u^m \ell(u)$ with $h,\ell\in\Bbbk[u]$.
Then \eqref{eq:Q-rec} becomes
\[
Q_G(u)=u^m\Bigl(A(u)+uB(u)\Bigr),
\]
where \(A(u)=-(r-1)h(u)-\ell(u)\) and \( B(u)=rh(u)+\ell(u)\).
Hence
\[
M(G)=m+\ord_{u=0}\bigl(A(u)+uB(u)\bigr).
\]
If $A(0)\neq 0$, then $\ord_{u=0}(A+uB)=0$, so $M(G)=m$. If $A(0)=0$, then $\ell(0)=-(r-1)h(0)$, and consequently
\[
B(0)=rh(0)+\ell(0)=rh(0)-(r-1)h(0)=h(0)\neq 0.
\]
Indeed, if $h(0)=0$, then $\ell(0)=0$ as well. So both $Q_H(u)=u^m h(u)$ and
$Q_L(u)=u^m \ell(u)$ are divisible by $u^{m+1}$, which implies
$M(H),M(L)\ge m+1$, contradicting $m=\min\{M(H),M(L)\}$.
Hence $\ord_{u=0}(A(u)+uB(u))=1$ in this case. Therefore,
\begin{equation}\label{eq:M-upper}
M(G)\ \le\ m+1\ =\ 1+\min\{M(H),M(L)\}.
\end{equation}

By the induction hypothesis (applied componentwise if $H$ or $L$ is disconnected),
we have $M(H)\le i(H)$ and $M(L)\le i(L)$. Combining \eqref{eq:i-rec} and \eqref{eq:M-upper},
\[
M(G)\ \le\ 1+\min\{M(H),M(L)\}\ \le\ 1+\min\{i(H),i(L)\}\ =\ i(G),
\]
which proves the first claim.

Hence
\[
\reg(R/J(G))-\deg h_{R/J(G)}(t)=M(G)-i(G)+1\le 1.\qedhere
\]
\end{proof}

\section{Possible 
$(\reg,\deg h)$-pairs and conjectural bounds}
\label{sec:possible-values}

Fix an integer $n\ge 2$ and let $G$ range over connected graphs on $n$ vertices. In this section, we  provide bounds on $\reg(R/J(G))$ and $ \deg h_{R/J(G)}$. Our general goal is to describe the set of pairs
\[
\Bigl(\reg(R/J(G)),\ \deg h_{R/J(G)}(t)\Bigr)
\]
that occur.

\begin{figure}[ht]
\centering
\begin{tikzpicture}[scale=0.65]
    \draw[->] (0,0) -- (8.5,0) node[right] {$\text{reg}$};
    \draw[->] (0,0) -- (0,8.5) node[above] {$\text{deg}$};

    \foreach \x in {1,2,3,4,5,6,7,8} {
        \draw (\x,-0.1) -- (\x,0.1);
        \node[below] at (\x,-0.2) {\x};
    }
    
    \foreach \y in {1,2,3,4,5,6,7,8} {
        \draw (-0.1,\y) -- (0.1,\y);
        \node[left] at (-0.1,\y) {\y};
    }

    \foreach \x/\y in {
        3/3, 4/3, 4/4, 4/5, 4/6, 4/7,
        5/4, 5/5, 5/6, 5/7,
        6/4, 6/5, 6/6, 6/7,
        7/5, 7/6, 7/7
    } {
        \filldraw[blue] (\x,\y) circle (3pt);
    }
\end{tikzpicture}
\caption{All possible $(\reg, \deg h)$ values for connected graphs on $n=9$ vertices.}
\label{fig:reg_deg_9_final}
\end{figure}
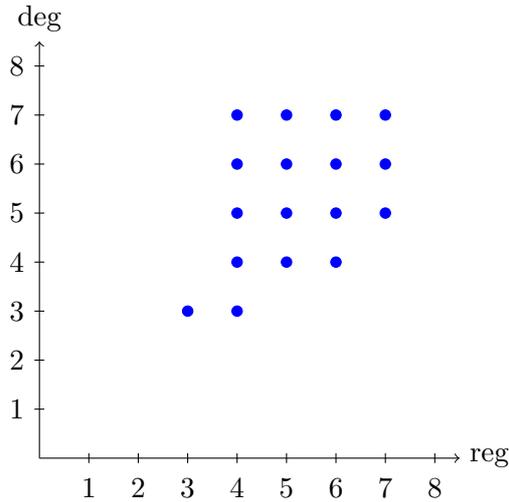

\begin{remark}
Let $G$ be a finite simple graph on $n$ vertices. Then  one has $\deg h_{R/J(G)}(t)\le n-2$ by Lemma~\ref{lem:HS1}.
Moreover, $\reg(R/J(G))=\pdim(R/I(G))-1$ and $\pdim(R/I(G))\le n-1$. So $\reg(R/J(G))\le n-2$.
Notice that both bounds are sharp.
\end{remark}

To relate $\deg h_{R/J(G)}(t)$ to combinatorial data, recall from
\Cref{thm:deg_edge_cover} and \Cref{lem:deg-via-ord} that 
\begin{equation}\label{eq:M-degree-identities}
M(G)=(n-2)-\deg h_{R/J(G)}(t)=\alpha(G)-\deg h_{R/I(G)}(t).
\end{equation}

In what follows, we record a convenient region in which both $\deg h_{R/J(G)}(t)$ and $\reg(R/J(G))$
are forced to be relatively large.

\begin{cor}\label{cor:alpha-small-forces-large}
If $\alpha(G)\le \left\lfloor\frac{n}{2}\right\rfloor+1$, then
\begin{enumerate}
    \item[(a)] \(\deg h_{R/J(G)}(t)  \ \ge\ \left\lceil\frac{n}{2}\right\rceil-2\),
    \item[(b)] \(\reg(R/J(G)) \ \ge\ \left\lceil\frac{n}{2}\right\rceil-2\).
\end{enumerate}
\end{cor}

\begin{proof}
(a) By Lemma~\ref{lem:M_bound}, we have
\[
M(G)\ \le\ \alpha(G)-1\ \le\ \left\lfloor\frac{n}{2}\right\rfloor.
\]
Using $M(G)=(n-2)-\deg h_{R/J(G)}(t)$ from \eqref{eq:M-degree-identities} gives
\[
\deg h_{R/J(G)}(t)=n-2-M(G)\ \ge\ n-2-\left\lfloor\frac{n}{2}\right\rfloor
=\left\lceil\frac{n}{2}\right\rceil-2.
\]

(b) For the regularity bound, note that $\reg(R/J(G))=\pdim(R/I(G))-1$. A standard lower bound
(see \cite[Corollary~3.33]{MoreyVillarreal}) gives
\[
\pdim(R/I(G))\ \ge\ \mathrm{bight}I(G),
\]
where $\mathrm{bight}I(G)$ is the maximum size of a minimal vertex cover. In particular,
$\mathrm{bight}I(G)\ge \tau(G)$, where $\tau(G)$ is the minimum size of a vertex cover.
Since $\tau(G)=n-\alpha(G)$, we obtain
\[
\pdim(R/I(G))\ \ge\ n-\alpha(G)
\ \ge\ n-\left(\left\lfloor\frac{n}{2}\right\rfloor+1\right)
=\left\lceil\frac{n}{2}\right\rceil-1,
\]
and therefore $\reg(R/J(G))=\pdim(R/I(G))-1\ge \left\lceil\frac{n}{2}\right\rceil-2$.
\end{proof}

The hypothesis $\alpha(G)\le \lfloor n/2\rfloor+1$ in Corollary~\ref{cor:alpha-small-forces-large}
cannot be removed: when $\alpha(G)$ is large, both $M(G)$ and $\pdim(R/I(G))$ can behave quite differently.
The following family (a special case of the split graphs from \Cref{sec:split-graphs}) is a convenient source
of examples.

\begin{remark}
Fix integers $p\ge 2$ and $q\ge 1$. Let $H_{p,q}$ be the graph obtained from a clique $K_p$
on vertices $x_1,\dots,x_p$ by attaching $q$ pendant leaves to each $x_i$.
Then $H_{p,q}$ is connected and chordal (indeed split), with $n=p(q+1)$ vertices. Then by \Cref{prop:split-multiplicity} and \Cref{cor:split-deg-h} (or, by \Cref{lem:q-whisker-indep-poly}), we have
\[
M(H_{p,q})=(p-1)q 
\qquad \text{ and } \qquad
\deg h_{R/J(H_{p,q})}(t)= p+q-2
\]
Moreover, one can verify that $i(H_{p,q})=1+(p-1)q$ so that
\[
\pdim(R/I(H_{p,q}))=n-i(H_{p,q})=p+q-1.
\]
Consequently,
\[
\Bigl(\reg(R/J(H_{p,q})),\ \deg h_{R/J(H_{p,q})}(t)\Bigr)=(p+q-2,\ p+q-2).
\]
In particular, $M(H_{p,q})$ can exceed $\lfloor n/2\rfloor$ when $p$ is large, showing that the
conclusion of Corollary~\ref{cor:alpha-small-forces-large} is genuinely tied to the hypothesis on $\alpha(G)$.

For instance, take $(p,q)=(4,4)$. Then $n=20$ and $\alpha(H_{4,4})=16>\lfloor n/2\rfloor+1=11$.
Here $\pdim(R/I(H_{4,4}))=p+q-1=7<\lceil n/2\rceil-1=9$, so the lower bound on $\reg(R/I(G))$
from Corollary~\ref{cor:alpha-small-forces-large} fails without the assumption on $\alpha(G)$.
\end{remark}

Using Macaulay2, we computed all possible pairs $(\reg,\deg h)$ for connected graphs up to $n=12$
(see Figure~\ref{fig:reg_deg_9_final} for $n=9$). These computations suggest that the realized points lie in a
band around the diagonal, leading to the following conjecture.

\begin{conj}\label{conj:bounds-reg-deg}
Let $G$ be a connected graph on $n$ vertices. Then
\[
\bigl|\reg(R/J(G))-\deg h_{R/J(G)}(t)\bigr|
\ \le\ \left\lceil\frac{n}{2}\right\rceil-2.
\]
\end{conj}

Not every pair $(r,d)\in \mathbb{Z}_{\ge 0}^2$ occurs as
\[
\bigl(\reg(R/J(G)),\ \deg h_{R/J(G)}(t)\bigr)
\]
for a connected graph $G$.  We record a simple obstruction.

\begin{lem}\label{lem:unrealizable}
There is no connected graph $G$ such that
\[
\bigl(\reg(R/J(G)),\ \deg h_{R/J(G)}(t)\bigr)=(1,d)
\]
with $d\ge 2$.
\end{lem}

\begin{proof}
Let $G$ be a connected graph on $n$ vertices.  Then we have
\[
1= \reg(R/J(G))=\pdim(R/I(G))-1 \implies \pdim(R/I(G))=2.
\]
Using the following lower bound from \cite[Proposition 4.7]{ProjDim}
\[
\pdim(R/I(G))\ \ge\ n-i(G),
\]
we obtain $i(G)\ge n-2$.

Let $\Delta(G)$ be the maximum vertex degree of $G$ and choose a vertex $v$ with $\deg(v)=\Delta(G)$.
Any maximal independent set containing $v$ can use at most the vertices outside $N(v)$, so it has size
at most $n-\Delta(G)$.  Hence $i(G)\le n-\Delta(G)$, and therefore $\Delta(G)\le 2$.
Since $G$ is connected, it follows that $G$ is either a path $P_n$ or a cycle $C_n$.

Finally, since $i(G)\le \alpha(G)$, the formulas $\alpha(C_n)=\lfloor n/2\rfloor$ and
$\alpha(P_n)=\lceil n/2\rceil$ yield $n\le 4$ in the cycle case and $n\le 5$ in the path case.
A direct check on these possible graphs shows that the only connected graphs with
$\reg(R/J(G))=1$ are $C_3$, $P_3$, and $P_4$, and in each case $\deg h_{R/J(G)}(t)=1$.
Thus no connected graph realizes $(1,d)$ with $d\ge 2$.
\end{proof}

More broadly, our computations suggest that unrealizable pairs $(r,d)$ are constrained by simple linear
inequalities.

\begin{conj}\label{conj:unrealizablepairs}
If either $\displaystyle r\leq \left\lceil \frac{d}{2}\right\rceil$ or $\displaystyle d\geq \left\lceil \frac{2r-1}{3}\right\rceil$,
then there is no connected graph $G$ such that
\[
\bigl(\reg(R/J(G)),\ \deg h_{R/J(G)}(t)\bigr)=(r,d).
\]
\end{conj}

\section{Conclusion and Future Directions}

A recurring goal throughout this paper has been to understand, for fixed $n$, which pairs
\[
\bigl(\reg(R/J(G)),\ \deg h_{R/J(G)}(t)\bigr)
\]
arise from connected graphs on $n$ vertices, and to identify graph families that populate large
portions of the feasible region. The results of \Cref{sec:rad2-trees,sec:split-graphs,sec:deg=reg-1,sec:reg=n-2,sec:block_graphs}
together with the computational data in \Cref{sec:possible-values} suggest that a surprisingly small
collection of structured classes already accounts for most observed pairs.

In particular, trees of radius at most two appear to realize most pairs on the diagonal $\deg h=\reg$
and above it, while our chordal constructions provide systematic families below the diagonal.
This motivates the following question.

\begin{question}\label{q:classes-realize}
Fix $n$.  Do there exist explicit graph classes (for instance, radius-$2$ graphs and/or chordal graphs)
whose members realize every pair
\[
\bigl(\reg(R/J(G)),\ \deg h_{R/J(G)}(t)\bigr)
\]
that occurs among connected graphs on $n$ vertices?
\end{question}

We hope that the invariant $M(G)$, together with its recursions and behavior under graph operations,
provides a useful framework for further study of these realizability problems.

\textbf{Acknowledgements.} The authors thank Tài Huy Hà for encouraging them to investigate cover ideals as a follow-up to \cite{biermann2024degree}. S.~Kara was supported by NSF Grant DMS-2418805.  Computations were performed on the Colgate Supercomputer, which is partially supported by NSF Grant OAC-2346664.

\bibliographystyle{abbrv}
\bibliography{ref}

@article {hibi2022regularity,
    AUTHOR = {Hibi, Takayuki and Kimura, Kyouko and Matsuda, Kazunori and
              Van Tuyl, Adam},
     TITLE = {The regularity and {$h$}-polynomial of {C}ameron-{W}alker
              graphs},
   JOURNAL = {Enumer. Comb. Appl.},
  FJOURNAL = {Enumerative Combinatorics and Applications},
    VOLUME = {2},
      YEAR = {2022},
    NUMBER = {3},
     PAGES = {Paper No. S2R17, 12},
      ISSN = {2710-2335},
   MRCLASS = {13A70 (05C70 05E40 13D02)},
  MRNUMBER = {4468202},
       DOI = {10.54550/eca2022v2s3r17},
       URL = {https://doi.org/10.54550/eca2022v2s3r17},
}

@article {hibi2018regularity,
    AUTHOR = {Hibi, Takayuki and Matsuda, Kazunori},
     TITLE = {Regularity and {$h$}-polynomials of monomial ideals},
   JOURNAL = {Math. Nachr.},
  FJOURNAL = {Mathematische Nachrichten},
    VOLUME = {291},
      YEAR = {2018},
    NUMBER = {16},
     PAGES = {2427--2434},
      ISSN = {0025-584X,1522-2616},
   MRCLASS = {05E40 (13F20 13H10)},
  MRNUMBER = {3884268},
MRREVIEWER = {Hassan\ Haghighi},
       DOI = {10.1002/mana.201700476},
       URL = {https://doi.org/10.1002/mana.201700476},
}

@article {hibi_binomial,
    AUTHOR = {Hibi, Takayuki and Matsuda, Kazunori},
     TITLE = {Regularity and {$h$}-polynomials of binomial edge ideals},
   JOURNAL = {Acta Math. Vietnam.},
  FJOURNAL = {Acta Mathematica Vietnamica},
    VOLUME = {47},
      YEAR = {2022},
    NUMBER = {1},
     PAGES = {369--374},
      ISSN = {0251-4184,2315-4144},
   MRCLASS = {05E40 (13C70 13H10)},
  MRNUMBER = {4406578},
MRREVIEWER = {Margherita\ Barile},
       DOI = {10.1007/s40306-021-00416-3},
       URL = {https://doi.org/10.1007/s40306-021-00416-3},
}

@article {hibi2019regularity,
    AUTHOR = {Hibi, Takayuki and Matsuda, Kazunori and Van Tuyl, Adam},
     TITLE = {Regularity and {$h$}-polynomials of edge ideals},
   JOURNAL = {Electron. J. Combin.},
  FJOURNAL = {Electronic Journal of Combinatorics},
    VOLUME = {26},
      YEAR = {2019},
    NUMBER = {1},
     PAGES = {Paper No. 1.22, 11},
      ISSN = {1077-8926},
   MRCLASS = {13D02 (05C69 05C70 05E40 13D40)},
  MRNUMBER = {3919621},
MRREVIEWER = {Oana\ Stefania\ Olteanu},
       DOI = {10.37236/8247},
       URL = {https://doi.org/10.37236/8247},
}

@article {favacchio2020regularity,
    AUTHOR = {Favacchio, Giuseppe and Keiper, Graham and Van Tuyl, Adam},
     TITLE = {Regularity and {$h$}-polynomials of toric ideals of graphs},
   JOURNAL = {Proc. Amer. Math. Soc.},
  FJOURNAL = {Proceedings of the American Mathematical Society},
    VOLUME = {148},
      YEAR = {2020},
    NUMBER = {11},
     PAGES = {4665--4677},
      ISSN = {0002-9939,1088-6826},
   MRCLASS = {13F65 (05E40 13A70 13D02 13D40 13P10 14M25)},
  MRNUMBER = {4143385},
MRREVIEWER = {Alessio\ D'Al\`i},
       DOI = {10.1090/proc/15126},
       URL = {https://doi.org/10.1090/proc/15126},
}

@book {herzog2011monomial,
    AUTHOR = {Herzog, J\"urgen and Hibi, Takayuki},
     TITLE = {Monomial ideals},
    SERIES = {Graduate Texts in Mathematics},
    VOLUME = {260},
 PUBLISHER = {Springer-Verlag London, Ltd., London},
      YEAR = {2011},
     PAGES = {xvi+305},
      ISBN = {978-0-85729-105-9},
   MRCLASS = {13D02 (05E40 13D40 13F55 13P10)},
  MRNUMBER = {2724673},
MRREVIEWER = {Rahim\ Zaare-Nahandi},
       DOI = {10.1007/978-0-85729-106-6},
       URL = {https://doi.org/10.1007/978-0-85729-106-6},
}

@article{jacques2005betti,
  title={The {B}etti numbers of forests},
  author={Jacques, Sean and Katzman, Mordechai},
  journal={arXiv preprint math/0501226},
  year={2005}
}

@article{ProjDim,
    AUTHOR = {Dao, Hailong and Schweig, Jay},
     TITLE = {Projective dimension, graph domination parameters, and
              independence complex homology},
   JOURNAL = {J. Combin. Theory Ser. A},
  FJOURNAL = {Journal of Combinatorial Theory. Series A},
    VOLUME = {120},
      YEAR = {2013},
    NUMBER = {2},
     PAGES = {453--469},
      ISSN = {0097-3165,1096-0899},
   MRCLASS = {05E45},
  MRNUMBER = {2995051},
MRREVIEWER = {Micha\l \ Adamaszek},
       DOI = {10.1016/j.jcta.2012.09.005},
       URL = {https://doi.org/10.1016/j.jcta.2012.09.005},
}

@incollection {terai2007,
    AUTHOR = {Terai, Naoki},
     TITLE = {Alexander duality in {S}tanley-{R}eisner rings},
 BOOKTITLE = {Affine algebraic geometry},
     PAGES = {449--462},
 PUBLISHER = {Osaka Univ. Press, Osaka},
      YEAR = {2007},
      ISBN = {978-4-87259-226-9},
   MRCLASS = {13F55 (05E99 13D02)},
  MRNUMBER = {2330484},
MRREVIEWER = {Christopher\ A.\ Francisco},
}

@article {biermann2024degree,
    AUTHOR = {Biermann, Jennifer and Kara, Selvi and O'Keefe, Augustine and
              Skelton, Joseph and Sosa Castillo, Gabriel},
     TITLE = {Degree of {$h$}-polynomials of edge ideals},
   JOURNAL = {J. Algebraic Combin.},
  FJOURNAL = {Journal of Algebraic Combinatorics. An International Journal},
    VOLUME = {62},
      YEAR = {2025},
    NUMBER = {1},
     PAGES = {Paper No. 2, 26},
      ISSN = {0925-9899,1572-9192},
   MRCLASS = {05E40},
  MRNUMBER = {4930041},
       DOI = {10.1007/s10801-025-01411-9},
       URL = {https://doi.org/10.1007/s10801-025-01411-9},
}

@incollection {MoreyVillarreal,
    AUTHOR = {Morey, Susan and Villarreal, Rafael H.},
     TITLE = {Edge ideals: algebraic and combinatorial properties},
 BOOKTITLE = {Progress in commutative algebra 1},
     PAGES = {85--126},
 PUBLISHER = {de Gruyter, Berlin},
      YEAR = {2012},
      ISBN = {978-3-11-025034-3},
   MRCLASS = {13F20 (05C10 05C25 05C65 05E40 13-02 13F55)},
  MRNUMBER = {2932582},
MRREVIEWER = {Louiza\ Fouli},
}

@book {Hibi_book,
    AUTHOR = {Hibi, Takayuki},
     TITLE = {Algebraic combinatorics on convex polytopes},
 PUBLISHER = {Carslaw Publications, Glebe},
      YEAR = {1992},
     PAGES = {x+164},
      ISBN = {1-875399-04-6},
   MRCLASS = {52-02 (05E45 13F55 57Q05)},
  MRNUMBER = {3183743},
}

@book {miller_thesis,
    AUTHOR = {Miller, Ezra Nathan},
     TITLE = {Resolutions and duality for monomial ideals},
      NOTE = {Thesis (Ph.D.)--University of California, Berkeley},
 PUBLISHER = {ProQuest LLC, Ann Arbor, MI},
      YEAR = {2000},
     PAGES = {136},
      ISBN = {978-0599-86011-7},
   MRCLASS = {99-05},
  MRNUMBER = {2701037},
       URL =
              {http://gateway.proquest.com/openurl?url_ver=Z39.88-2004&rft_val_fmt=info:ofi/fmt:kev:mtx:dissertation&res_dat=xri:pqdiss&rft_dat=xri:pqdiss:9979735},
}

@article {levit2003roots,
    AUTHOR = {Levit, Vadim E. and Mandrescu, Eugen},
     TITLE = {On the roots of independence polynomials of almost all very
              well-covered graphs},
   JOURNAL = {Discrete Appl. Math.},
  FJOURNAL = {Discrete Applied Mathematics. The Journal of Combinatorial
              Algorithms, Informatics and Computational Sciences},
    VOLUME = {156},
      YEAR = {2008},
    NUMBER = {4},
     PAGES = {478--491},
      ISSN = {0166-218X,1872-6771},
   MRCLASS = {05C69 (05C05)},
  MRNUMBER = {2379079},
       DOI = {10.1016/j.dam.2006.06.016},
       URL = {https://doi.org/10.1016/j.dam.2006.06.016},
}

@article {CutlerKahl2016,
    AUTHOR = {Cutler, Jonathan and Kahl, Nathan},
     TITLE = {A note on the values of independence polynomials at {$-1$}},
   JOURNAL = {Discrete Math.},
  FJOURNAL = {Discrete Mathematics},
    VOLUME = {339},
      YEAR = {2016},
    NUMBER = {11},
     PAGES = {2723--2726},
      ISSN = {0012-365X,1872-681X},
   MRCLASS = {05C31 (05C69)},
  MRNUMBER = {3518424},
       DOI = {10.1016/j.disc.2016.05.019},
       URL = {https://doi.org/10.1016/j.disc.2016.05.019},
}

@article {ChudnovskySeymour2007,
    AUTHOR = {Chudnovsky, Maria and Seymour, Paul},
     TITLE = {The roots of the independence polynomial of a clawfree graph},
   JOURNAL = {J. Combin. Theory Ser. B},
  FJOURNAL = {Journal of Combinatorial Theory. Series B},
    VOLUME = {97},
      YEAR = {2007},
    NUMBER = {3},
     PAGES = {350--357},
      ISSN = {0095-8956,1096-0902},
   MRCLASS = {05C69},
  MRNUMBER = {2305888},
MRREVIEWER = {Steven\ D.\ Noble},
       DOI = {10.1016/j.jctb.2006.06.001},
       URL = {https://doi.org/10.1016/j.jctb.2006.06.001},
}

@article {brown2000roots,
    AUTHOR = {Brown, J. I. and Dilcher, K. and Nowakowski, R. J.},
     TITLE = {Roots of independence polynomials of well covered graphs},
   JOURNAL = {J. Algebraic Combin.},
  FJOURNAL = {Journal of Algebraic Combinatorics. An International Journal},
    VOLUME = {11},
      YEAR = {2000},
    NUMBER = {3},
     PAGES = {197--210},
      ISSN = {0925-9899,1572-9192},
   MRCLASS = {05C69},
  MRNUMBER = {1771611},
MRREVIEWER = {E.\ J.\ Farrell},
       DOI = {10.1023/A:1008705614290},
       URL = {https://doi.org/10.1023/A:1008705614290},
}

@article {peters2019conjecture,
    AUTHOR = {Peters, Han and Regts, Guus},
     TITLE = {On a conjecture of {S}okal concerning roots of the
              independence polynomial},
   JOURNAL = {Michigan Math. J.},
  FJOURNAL = {Michigan Mathematical Journal},
    VOLUME = {68},
      YEAR = {2019},
    NUMBER = {1},
     PAGES = {33--55},
      ISSN = {0026-2285,1945-2365},
   MRCLASS = {05C31 (05C69 37F10 82B20)},
  MRNUMBER = {3934603},
MRREVIEWER = {Helin\ Gong},
       DOI = {10.1307/mmj/1541667626},
       URL = {https://doi.org/10.1307/mmj/1541667626},
}

@article {brown2004location,
    AUTHOR = {Brown, J. I. and Hickman, C. A. and Nowakowski, R. J.},
     TITLE = {On the location of roots of independence polynomials},
   JOURNAL = {J. Algebraic Combin.},
  FJOURNAL = {Journal of Algebraic Combinatorics. An International Journal},
    VOLUME = {19},
      YEAR = {2004},
    NUMBER = {3},
     PAGES = {273--282},
      ISSN = {0925-9899,1572-9192},
   MRCLASS = {05C69},
  MRNUMBER = {2071474},
MRREVIEWER = {E.\ J.\ Farrell},
       DOI = {10.1023/B:JACO.0000030703.39946.70},
       URL = {https://doi.org/10.1023/B:JACO.0000030703.39946.70},
}

@article {bencs2018trees,
    AUTHOR = {Bencs, Ferenc},
     TITLE = {On trees with real-rooted independence polynomial},
   JOURNAL = {Discrete Math.},
  FJOURNAL = {Discrete Mathematics},
    VOLUME = {341},
      YEAR = {2018},
    NUMBER = {12},
     PAGES = {3321--3330},
      ISSN = {0012-365X,1872-681X},
   MRCLASS = {05C31 (05C05)},
  MRNUMBER = {3862630},
MRREVIEWER = {Helin\ Gong},
       DOI = {10.1016/j.disc.2018.06.033},
       URL = {https://doi.org/10.1016/j.disc.2018.06.033},
}

@inproceedings {split_graphs,
    AUTHOR = {Foldes, St\'ephane and Hammer, Peter L.},
     TITLE = {Split graphs},
 BOOKTITLE = {Proceedings of the {E}ighth {S}outheastern {C}onference on
              {C}ombinatorics, {G}raph {T}heory and {C}omputing ({L}ouisiana
              {S}tate {U}niv., {B}aton {R}ouge, {L}a., 1977)},
    SERIES = {Congress. Numer.},
    VOLUME = {No. XIX},
     PAGES = {311--315},
 PUBLISHER = {Utilitas Math., Winnipeg, MB},
      YEAR = {1977},
      ISBN = {0-919628-19-2},
   MRCLASS = {05C99},
  MRNUMBER = {505860},
}

\end{document}